\documentclass[a4paper,12pt]{amsart}
\usepackage{amsmath,amssymb}
\usepackage[left=2cm,right=2cm]{geometry}
\makeatletter
\@namedef{subjclassname@2020}{%
	\textup{2020} Mathematics Subject Classification}
\makeatother
\newtheorem{thm}{Theorem}
\newtheorem{claim}[thm]{Claim}

\newtheorem{prop}[thm]{Proposition}
\newtheorem{lemma}[thm]{Lemma}
\newtheorem{cor}[thm]{Corollary}

\newtheorem{lem}[thm]{Lemma}
\theoremstyle{remark}\newtheorem{rem}[thm]{Remark}

\def\C{\Bbb C}

\def\hol{\mathcal O}

\def\O{\Omega}

\def\k{\kappa}

\def\ov{\overline}

\def\lk{ l^{\kappa}}
\def\Cn{\mathbb{C}^N}
\def\ds{\displaystyle}
\def\re{\operatorname{Re}}
\def\im{\operatorname{Im}}

\begin{document}
	
	\title[Precise estimates of invariant distances]
	{Precise estimates of invariant distances on strongly pseudoconvex domains}
	
	\author{{\L}ukasz Kosi{\'n}ski, Nikolai Nikolov, and Ahmed Yekta {\"O}kten}
	\address{{\L}. Kosi{\'n}ski\\Institute of Mathematics, Faculty of Mathematics
		and Computer Science, Jagiellonian University \\ Lojasiewicza 6, 30-348 \\
		Krakow, Poland}
	\email{lukasz.kosinski@uj.edu.pl}
	\address{N. Nikolov\\
		Institute of Mathematics and Informatics\\
		Bulgarian Academy of Sciences\\
		Acad. G. Bonchev Str., Block 8\\
		1113 Sofia, Bulgaria}
	\address{Faculty of Information Sciences\\
		State University of Library Studies
		and Information Technologies\\
		69A, Shipchenski prohod Str.\\
		1574 Sofia, Bulgaria}
	\email{nik@math.bas.bg}
	\address{A.Y.{\"O}kten\\
		Dipartimento di Matematica \\
		Universit\`a di Roma “Tor Vergata”\\
	Via della
	Ricerca Scientifica 1, 00133 Roma, Italy} \email{okten@mat.uniroma2.it}
	
	\thanks{}
	
	
	
	\maketitle
	
	\begin{center}
		\textit{Dedicated to Prof. Dr. Peter Pflug on the occasion of his 80th birthday}
	\end{center}
	
	\begin{abstract}
		Studying the behavior of real and complex geodesics we provide sharp estimates for the Kobayashi distance, the Lempert function, and the Carath{\'e}odory distance on strongly pseudoconvex domain with $\mathcal{C}^{2,\alpha}$-smooth boundary. Similar estimates are also provided for the Bergman distance on strongly pseudoconvex domains with $\mathcal{C}^{3,1}$-smooth boundary.
	\end{abstract}
	
	\tableofcontents
	\setcounter{tocdepth}{1}
	\addtocontents{toc}{\protect\setcounter{tocdepth}{1}}
	\section{Introduction}

	In the last decades, the metric geometric properties of the Kobayashi distance have been explored and have led to many fruitful applications. This line of inquiry started with the work of Balogh and Bonk \cite{BB} where they proved the Gromov hyperbolicity of the Kobayashi distance on strongly pseudoconvex domains to reobtain the well-known results about extensions of biholomorphic maps between strongly pseudoconvex domains. Their results follow from the key estimate
	$$ k_\O(z,w)=g_\O(z,w)+\mathcal{O}(1),$$ where $k_\O$ is the Kobayashi distance of the strongly pseudoconvex domain $\O$ and $g_\O$ is a function obtained from the Carnot-Carath\'eodory metric on $\partial\O$. The downside of the estimate of Balogh-Bonk is that it is precise only when the distances are large. Nevertheless, that was sufficient for their purposes as  Gromov hyperbolicity is only relevant for geometry in the large.
	
	The goal of this paper is to provide upper and lower bounds for the Kobayashi distance on strongly pseudoconvex domains with $\mathcal{C}^{2,\alpha}$-smooth boundary in terms of simple quantities coming from the Euclidean geometry of the domain. Our estimates remain sensitive even when the Kobayashi distance is small. Precisely, we show that there are universal constants $0<c<C$ (depending only on a domain) such that $$\log\left(1+c{A_\O(z,w)}\right)\leq k_\O(z,w) \leq \log\left(1+CA_\O(z,w)\right),
	$$ that is $$ c A_\O \leq e^{k_\O} - 1 \leq C A_\O,$$ where $$A_\O(z,w)=\dfrac{\|(z-w)_z\|+\|z-w\|^2+\|z-w\|\delta^{1/2}_\O(z)}{ {\delta^{1/2}_\O(z)\delta^{1/2}_\O(w)}}.$$
	Above $\delta_\Omega(x):=\displaystyle{\inf\{\lVert x-y \rVert: y\in\partial\Omega\}}$ is the distance to $\partial \O$, and $v_z$ is the complex normal component of the vector $v$ taken with respect to the closest point to $z$ on $\partial\Omega$, denoted throughout the paper by $\pi_\O(z)$; and $v_z:= \langle v, \eta_{\pi(z)} \rangle$, where $\eta_{p}$ denote the outer unit normal vector to $\partial \O$ at $p$
	and $\langle z,w \rangle := \sum_{j=1}^N z_j \bar w_j$. If such a closest point is not unique, the complex normal may be taken with respect to any of them.
	
	Our estimate encapsulates the one of Balogh and Bonk. Furthermore, due to the earlier comparison results in \cite{NT2}, this estimate extends to the Lempert function and the Carath{\'e}odory distance; and to the Bergman distance as well, provided that higher boundary regularity assumptions are satisfied.
	
	The estimates we obtain for the Kobayashi distance rely on the previous work, namely the estimate \cite[Corollary 8]{NA} for the upper bound, and the estimates \cite[Proposition 1.6, Proposition 1.7]{NT} and \cite[Proposition 2]{KN} for the lower bound. Actually, we shall observe that they remain critical only when the points $z,w\in\O$ tend to a $p\in\partial\O$ "tangentially" (see Section~\ref{transformations} for the definition). In fact, some of our results are a natural continuation of \cite{KN}. In particular, we prove a conjecture formulated there to extend \cite[Proposition 2]{KN} to the "tangential" case.
	
	Notably, we also complement some results of \cite{KN} regarding the Euclidean behavior of real and complex geodesics of strongly pseudoconvex domains. We compare the lengths of real geodesics with their maximal boundary distance and describe Euclidean diameters of complex geodesics in terms of the boundary distance of the points which define them and of the 'angle' the complex line defined by the two points makes with the boundary of the domain. In particular, we improve a result of Huang \cite[Corollary 1.2]{H3} (see also \cite[Proposition 7]{KN} for more precise estimates) about the behavior of extremal maps for the Kobayashi-Royden pseudometric on strongly pseudoconvex domains.
	
	Our approach  works in the $\mathcal{C}^{2,\alpha}$-smooth case, where $\alpha>0$. The reason for this is that some methods rely heavily on the theory developed by Lempert \cite{L1, L2, L} and extended later by Huang \cite{H, H2, H3}.
	
	This paper is organized as follows. In Section~\ref{pre} we recall the basic concepts related to invariant distances and functions. Next, in Section~\ref{sectionresults}, we present our main results, and in Section~\ref{sectionapplications}, we discuss their possible applications. Section~\ref{transformations} contains the main tools that will be used in the paper. We will describe here the key details of our approach.
	The proofs are split into three parts. In Sections~\ref{sectiongeometricproofs} and~\ref{sec:geometricproofs} we present the ones concerning the behavior of real and complex geodesics on strongly pseudoconvex domains. Building on those, in Section~\ref{estimateproofs}, we establish our estimates of invariant distances.
	Finally, in Section~\ref{strictlylinearlyconvexcase}, we study the basic concepts about the Euclidean geometry of strongly pseudoconvex domains and show that our estimates can be simplified in the strictly linearly convex case.

	\section{Basic definitions}\label{pre}
	
	Let $\Omega$ be a domain in $\C^N$, $z,w\in \Omega$ and $v\in\Cn$. Recall that the Carath{\'e}odory distance and the Lempert function on $\O$ are given as $$ c_\O(z,w):=\tanh^{-1} \tilde{c}_\O(z,w) \:\:\:\:\: \text{and} \:\:\:\:\: l_\O(z,w):=\tanh^{-1} \tilde{l}_\O(z,w),$$
	where $\Delta$ is the
	unit disc, $\tilde{c}_\Omega(z,w):=\sup\{|\alpha|:\exists f\in\hol(\O,\Delta)
	\hbox{ with }f(z)=0,f(w)=\alpha\}$ and $\tilde{l}_\Omega(z,w):=\inf\{|\alpha|:\exists\varphi\in\hol(\Delta,\O)
	\hbox{ with }\varphi(0)=z,\varphi(\alpha)=w\} $. In general, the Lempert function does not satisfy the triangle inequality. The Kobayashi pseudodistance $k_\Omega$ is defined as the largest pseudodistance
	which does not exceed the Lempert function. Thus in general $c_\O(z,w)\leq k_\O(z,w) \leq l_\O(z,w)$, where $z,w\in \O$. The celebrated Lempert theorem asserts that all these functions coincide if $\O$ is convex (\cite{L1}).
	
	Those notions have infinitesimal counterparts. For a point $z$ and a vector $v\in \C^N$, the Kobayashi-Royden pseudometric is given by
	$$\k_\O(z;v)=\inf_{\alpha\in\mathbb C}\{|\alpha|:\exists\varphi\in\O(\Delta,\O)\hbox{ with }
	\varphi(0)=z,\alpha\varphi'(0)=v\},$$ and the Kobayashi-Royden length of an absolutely continuous curve $\gamma:I\rightarrow \O$ is defined as $$ \lk_\O(\gamma):=\int_{I} \k_\O(\gamma(t);\gamma'(t))dt. $$
	
	It follows from \cite[Theorem 1.2]{V} that $k_\Omega$ is an integrated form of the Kobayashi-Royden pseudometric, that is
	$ k_\O(z,w)=\inf\lk_\O(\gamma)$, where the infimum is taken over all absolutely continuous curves joining $z$ to $w$.
	
	It is well-known that if a domain $\O$ is strongly pseudoconvex, then it is complete hyperbolic, i.e. $(\O,k_\O)$ forms a complete metric space. If this is the case, it follows from the well-known Hopf-Rinow theorem \cite[Proposition 3.7]{BH} that $(\O,k_\O)$ is a geodesic space, i.e. any $z,w\in\O$, can be joined by a (real) Kobayashi geodesic, i.e. an isometry between an interval and the Kobayashi distance of $\O$.  The holomorphic analogue of real geodesics are complex geodesics, which are holomorphic maps $\varphi:\Delta\rightarrow\O$ that are isometries with respect to the Kobayashi distance.
	
	In general, real Kobayashi geodesics need not exist, but there are substitutes. An absolutely continuous curve $\gamma:I\rightarrow\O$ is said to be a $(\lambda,\epsilon)$-geodesic if for any $t_1\leq t_2\in I$ we have
	$$ \lk_\O(\gamma|_{[t_1,t_2]}) \leq \lambda k_\O(\gamma(t_1),\gamma(t_2)) + \epsilon.$$
	
	It is noted in \cite{NO} that any $(\lambda,\epsilon)$-geodesic can be parametrized as a $(\lambda,\epsilon)$-almost geodesic in the sense of \cite[Definition 4.1]{BZ}. In particular, any $(\lambda,\epsilon)$-geodesic is a quasi-geodesic in the sense of \cite[Definition 4.7]{BZ}.
	
	Finally recall that the Bergman distance $b_\O(z,w)$ is the inner distance obtained by the Bergman metric  $$\beta_\Omega(z;v):=\dfrac{M_\O(z;v)}{K^{1/2}_\O(z)}, \:\:\:\:\: z\in \O, \:\: v \in \Cn,$$ where $$M_\O(z;v)=\sup\{\|df_z v\|: f \in L^2_h(\O), \|f\|_\O\leq 1, f(z)=0\}$$ and $K_\O$ is the Bergman kernel of $\O$ on the diagonal. This metric stems from a Hermitian form.  We refer the reader to \cite[section 12.7]{JP} for details.
	
	\section{Statement of the results}\label{sectionresults}
	Our main result is the following:
	\begin{thm}\label{profnikolovsdream}
		Let $\O$ be a strongly pseudoconvex domain with  $\mathcal{C}^{2,\alpha}$-smooth boundary, $\alpha\in (0,1]$ and let $d_\O$ be either the Carath\'eodory distance, the Kobayashi distance or the Lempert function on $\O$. Then, there exist constants $0<c<C$ (depending only on $\Omega$) such that
		\begin{equation*}\log\left(1+c {A_\O(z,w)}\right) \leq  d_\O(z,w) \leq  \log\left(1+C {A_\O(z,w)}\right) , \:\:\:z,w\in \O.
		\end{equation*}
		If additionally $\O$ has $\mathcal C^{3,1}$-smooth boundary, then the above assertion holds also when $d_\O$ is the Bergman distance on $\O$.
	\end{thm}
		
	The estimates given in Theorem~\ref{profnikolovsdream} will be proved for the Kobayashi distance first. Then, combined with the comparison results in \cite{NT2}, we will derive the assertions for the Carath{\'e}odory distance as well as the Lempert function, and also for the Bergman distance.
	
	\medskip
	
	We will show in Section ~\ref{strictlylinearlyconvexcase} that if $\O$ is strictly linearly convex, the term $\|z-w\|^2$ in $A_\O(z,w)$ in the estimates from Theorem \ref{profnikolovsdream} is superfluous.
	
	\medskip
	
	The lower bound in Theorem \ref{profnikolovsdream} for the Kobayashi distance $k_\O$ will be derived from Proposition \ref{normallowerbound}, stated below, in combination with the estimate provided in \cite[Theorem 1.6, Theorem 1.7]{NT}.
	\begin{prop}\label{normallowerbound}
		Let $\O$ be a strongly pseudoconvex domain with  $\mathcal{C}^{2,\alpha}$-smooth boundary. Then, there exists a constant $c>0$ such that
		$$ \log\left(1+\dfrac{c\|(z-w)_z\|}{\delta^{1/2}_\O(z)\delta^{1/2}_\O(w)}\right)\leq k_\Omega(z,w), \:\:\:z,w\in \O .$$
	\end{prop}
	A partial case of the above result was proved in \cite[Proposition 2]{KN}.
	
	\medskip
	Let us introduce some additional notation. For functions $f,g$ taking non-negative values on a given set $X$, we shall write $f \gtrsim g$ or $g \lesssim f$ if there exists a $c>0$ such that $ f(x) \geq c g(x)$ for all $x\in X$. The meaning of $f\sim g$ is analogous: it says that $f\gtrsim g$ and $f \lesssim g$.
	
	\medskip
	
	In order to prove Proposition \ref{normallowerbound} we will need several additional results, interesting on their own, about the behavior of real and complex geodesics of strongly pseudoconvex domains.

	\begin{prop}\label{quantitativevisibility}
		Let $\O$ be a strongly pseudoconvex domain with  $\mathcal{C}^{2,\alpha}$-smooth boundary. If $z,w\in\Omega$ and $\gamma_{z,w}:I\rightarrow \Omega$ is a real Kobayashi geodesic joining $z$ to $w$, then
		\begin{equation}\label{mainresultequation}
			D^{1/2}_\O(\gamma_{z,w}) \gtrsim	l(\gamma_{z,w}),
		\end{equation}
		where $D_\O(\gamma_{z,w}):=\max\{\delta_\Omega(\gamma(t)):t\in I\}$ and $l(\gamma_{z,w})$ denotes the Euclidean length of $\gamma(I)$. The
		constants implicit in the estimate above depend only on $\O$.
	\end{prop}
	The behavior of the geodesics of the unit ball
	shows that the above result is sharp. In particular, the power $1/2$ of the maximum boundary distance cannot be improved.
	
	Let us point out that Proposition \ref{quantitativevisibility} is an improvement of the first part of \cite[Corollary 12]{NO}. However, the technique there is entirely different from the one applied in our paper. It should also be mentioned that the first part of \cite[Corollary 12]{NO} is valid if the boundary is just $\mathcal{C}^2$-smooth, while it is not clear whether Proposition~\ref{quantitativevisibility} remains true under lower boundary regularity assumptions.
	
	Throughout the paper, let $d_e(.)$ denote the Euclidean diameter of a set.
	
	Using Proposition \ref{quantitativevisibility} we will also prove the following result, conjectured in \cite{KN}.
	
	\begin{prop}\label{themissingpart}
		Let $\O$ be a strongly pseudoconvex domain with $\mathcal{C}^{2,\alpha}$-smooth boundary. For any complex geodesic $\varphi:\Delta\rightarrow\Omega$ parametrized so that
		$$D_\O(\varphi):=\max_{\zeta\in\Delta}\delta_{\Omega}(\varphi(\zeta))=\delta_\O(\varphi(0))$$ the following uniform estimates hold:
		\begin{equation}\label{nottrivial} D^{\frac{1}{2}}_\Omega(\varphi)  \sim
			d_e(\varphi (\Delta))\sim \max_{z\in\Delta}\|\varphi'(z)\| .
		\end{equation}
	\end{prop}
	Building upon this estimate, a stronger result that precisely describes the behavior of complex geodesics will be presented. For the class of domains
	we are considering, it complements \cite[Theorem 3]{KN}.
	
	\begin{thm}\label{profnikolovsconjecture}
		Let $\O$ be a strongly pseudoconvex domain with $\mathcal{C}^{2,\alpha}$-smooth boundary. For any complex geodesic $\varphi:\Delta\rightarrow\O$
		and $z,w\in\varphi(\Delta)$ we have
		\begin{equation}
			\dfrac{\|(z-w)_z\|}{\|z-w\|}+\delta^{1/2}_\O(z)+\|z-w\| \sim d_e(\varphi)
		\end{equation}
	\end{thm}
	
	Note that since $\O$ is a strongly pseudoconvex domain with $\mathcal C^{2,\alpha}$-smooth boundary, the complex geodesics $\varphi:\Delta\rightarrow\O$ extend $\mathcal C^{1,\alpha}$-smoothly to the closed unit disc (see \cite[Lemma 3.1]{L}). Thus, as the infinitesimal version of Theorem \ref{profnikolovsconjecture} we obtain:
	
	\begin{cor}
		Let $\O$ be a strongly pseudoconvex domain with $\mathcal{C}^{2,\alpha}$-smooth boundary, $\varphi:\Delta\rightarrow \O$ be a complex geodesic
		which extends continuously to the closed disc
		and $\zeta \in \overline \Delta$.
		Then
		$$ \dfrac{\|(\varphi'(\zeta))_{\varphi(\zeta)}\|}{\|\varphi'(\zeta)\|} + \delta^{1/2}_\O(\varphi(\zeta)) \sim d_e(\varphi),\:\:\:\:\:\zeta\in\Delta,$$
		and
		$$ \dfrac{\|(\varphi'(\zeta))_{\varphi(\zeta)}\|}{\|\varphi'(\zeta)\|} \sim d_e(\varphi),\:\:\:\:\:\zeta\in\partial \Delta.$$
	\end{cor}
	
	This result describes the angle of intersection of the boundary of the complex geodesic and the boundary of the domain. It gives a converse to the first statement of \cite[Proposition 7]{KN}, and is an improvement of \cite[Corollary 1.2]{H3} and \cite[Theorem 2]{H2}.

	\begin{rem} Switching $z,w$ in the estimates in Theorem~\ref{profnikolovsdream} one can immediately see that $A_\O(z,w)\sim A_\O(w,z)$. This \textit{symmetry} of $A_{\O}$ can be deduced directly, even in the case when $\O$ has $\mathcal C^{1,1}$-smooth boundary, thus allowing for lower regularity settings.
		
		To see this set $\tilde \delta_\O$ to be the signed boundary distance function, i.e. $\tilde \delta_\O=-\delta_\O$ on $\ov \O$ and $\tilde \delta_\O=\delta_\O$ on $\Cn\setminus\O$, $g(z)=2 \ov \partial \tilde \delta_\O(z)$ and let $\eta_{\pi(z)}$ denote the outer unit normal vector to $\partial \O$ at $\pi_\O(z)$. It is well-known that $\eta_{\pi_\O(z)}= g(\pi_\O(z)) = g(z)$. If $\O$ has $\mathcal{C}^{1,1}$-smooth boundary, $g$ is Lipschitz, so $$ \|(z-w)_z-(z-w)_w\|=|\langle z-w, g(z)-g(w)\rangle| = \mathcal O(\|z-w\|^2).$$
		Furthermore, assuming that $\delta_\O(w)\leq\delta_\O(z)$ we have  $\delta_\O(z)-\delta_\O(w)=\|(z-w)_z^{\mathbb R}\|+\mathcal{O}(\|z-w\|^2),$ where
		$v_z^{\mathbb R}:= \Re (\langle v, \eta_{\pi(z)} \rangle) \eta_{\pi(z)}$ is the real normal component of the vector $v$ taken with respect to $\pi_\O(z)$.
		Consequently, as $\|(z-w)^{\mathbb{R}}_z\|\leq\|(z-w)_z\|$, we have \begin{multline*}
			\|z-w\|\delta^{1/2}_\O(z)\lesssim \|z-w\|\delta^{1/2}_\O(w)+\|z-w\|\|(z-w)_z^{\mathbb R}\|^{1/2}+\mathcal{O}(\|z-w\|^2)
			\\ \lesssim  \|z-w\|\delta^{1/2}_\O(w)+\|(z-w)_z\|+\mathcal{O}(\|z-w\|^2).
		\end{multline*}
		In the last inequality we used the trivial fact that $\|z-w\| \|(z-w)_z^{\mathbb R}\|^{1/2} \leq \|z-w\|^2 +\|(z-w)_z\|$. 	
		
		Combining the two estimates achieved above, one easily infers that $A_\O(z,w)\sim A_\O(w,z)$.
	\end{rem}

	\section{Applications, and relationship to previous work}\label{sectionapplications}
	We will now present some applications of our results, and how some classical estimates can be
	derived from them as corollaries.
	
	We can localize the estimates given in Theorem~\ref{profnikolovsdream} for the Kobayashi distance as well as for the Lempert function. Precisely, the following holds:
	
	\begin{cor}\label{localizationforkobayashiandlempert}
		Let $\O$ be a domain in $\mathbb{C}^N$ and $p\in\partial\O$ be a $\mathcal{C}^{2,\alpha}$-smooth strongly pseudoconvex boundary point. Then, there exist constants $0<c<C$ such that
		\begin{equation*}\log\left(1+c{A_\O(z,w)}\right) \leq  k_\O(z,w)\leq
			l_\O(z,w)\leq  \log\left(1+C{A_\O(z,w)}\right),
	\end{equation*}\end{cor} \noindent for $z,w\in \O$ near $p$.
	
	We will provide a proof of this Corollary in Section~\ref{estimateproofs}. Let us mention that localization for the Carath{\'e}odory distance seems to require global assumptions, see \cite[Theorem 1.6.]{NT2}.
	
	Recall the estimate of Balogh and Bonk:
	
	\begin{thm}\cite[Corollary 1.3]{BB}\label{baloghbonkresult}
		Let $\O$ be a strongly pseudoconvex domain with $\mathcal C^2$-smooth boundary. There exists a constant $C>0$ depending on $\O$ such that
		$$ g_\O(z,w)-C \leq k_\O(z,w) \leq g_\O(z,w)+C .$$
		Here
		$$g_\O(z,w)=\log\left(\dfrac{d^H_\O(\pi_\O(z),\pi_\O(w))^2+\max\{\delta_\O(z),\delta_\O(w)\}}{\delta^{1/2}_\O(z)\delta^{1/2}_\O(w)}\right),$$ where $d^H_\O(\cdot, \cdot)$ is the distance obtained from the Carnot-Carath{\'e}odory metric on $\partial\O$.
	\end{thm}	
	
	\begin{prop}
		\label{profnikolovsclaim}
		Let $\O$ be a strongly pseudoconvex domain with $\mathcal C^{2,\alpha}$-smooth boundary. Then, the estimates in Theorem \ref{baloghbonkresult} follow from Theorem \ref{profnikolovsdream}.
	\end{prop}
	
	\begin{proof}[Proof of Proposition \ref{profnikolovsclaim}]
		The classical Box-Ball estimate recalled in \cite[Proposition 3.1]{BB} states that
		\begin{equation}\label{ballboxconsequence}
			d^H_\O(\pi(z),\pi(w))^2 \sim {\|(\pi_\O(z)-\pi_\O(w))_{\pi_\O(z)}\|+\|\pi_\O(z)-\pi_\O(w)\|^2} .
		\end{equation}
		Thus,  \begin{multline*}
			d^H_\O(\pi(z),\pi(w))^2+\max\{\delta_\O(z),\delta_\O(w)\} \sim \\  \|(\pi_\O(z)-\pi_\O(w))_{\pi_\O(z)}\|+\|\pi_\O(z)-\pi_\O(w)\|^2+\max\{\delta_\O(z),\delta_\O(w)\}.
		\end{multline*}
		
		Theorem \ref{profnikolovsdream} implies that
		\begin{equation}\label{uptoadditiveconstant} k_\O(z,w)=\log\left(1+{ A_\O(z,w)}\right)+\mathcal{O}(1) .\end{equation}
		
		We claim that \begin{multline} \label{whydoesourestimateleadtobb}
			(A_\O(z,w)+1)\delta^{1/2}_\O(w)\delta^{1/2}_\O(w) \sim \\
			\|(\pi_\O(z)-\pi_\O(w))_{\pi_\O(z)}\|+\|\pi_\O(z)-\pi_\O(w)\|^2+\max\{\delta_\O(z),\delta_\O(w)\}.\end{multline}
		
		Without loss of generality assume that $\delta_\O(w)\leq \delta_{\O}(z)$. Observe that
		$\|\pi(z)-\pi(w)\|\leq \|z-w\|+\|\pi(z)-z\|+\|\pi(w)-w\|$, whence $ |\|\pi(z)-\pi(w)\|-\|z-w\|| \leq 2 \delta_\O(z)$. This implies that
		\begin{multline*}
			\|(\pi_\O(z)-\pi_\O(w))_{\pi_\O(z)}\|+\|\pi_\O(z)-\pi_\O(w)\|^2+\max\{\delta_\O(z),\delta_\O(w)\} \sim \\ \|(z-w)_z\|+\|z-w\|^2 + \max\{\delta_\O(z),\delta_\O(w)\} .
		\end{multline*}
		It remains to show that
		\begin{equation}\label{fa}\max\{\delta_\O(z),\delta_\O(w)\}\lesssim (A_\O(z,w) +1)\delta^{1/2}_\O(z)\delta^{1/2}_\O(w)
		\end{equation} and
		\begin{equation}\label{sa} \|z-w\|\delta^{1/2}_\O(z)+\delta^{1/2}_\O(z)\delta^{1/2}_\O(w) \lesssim \|(z-w)_z\|+\|z-w\|^2 + \max\{\delta_\O(z),\delta_\O(w)\}.
		\end{equation}
		
		To see \eqref{fa} note that $\max\{\delta_\O(z),\delta_\O(w)\}-\min\{\delta_\O(z),\delta_\O(w)\}= \delta_\O(z)-\delta_\O(w)=\|(z-w)_z^{\mathbb{R}}\|+\mathcal{O}(\|z-w\|^2)$.
		As $\|(z-w)_z^{\mathbb{R}}\|\leq \|(z-w)_z\|$ we have $$\delta_\O(z) \leq \delta_\O(w) + \|(z-w)_z\| + C\|z-w\|^2 \leq  \delta^{1/2}_\O(z)\delta^{1/2}_\O(w) + \|(z-w)_z\| + C\|z-w\|^2,$$ and \eqref{fa} follows.
		
		To prove \eqref{sa}  it is enough to observe that $2\|z-w\|\delta^{1/2}_\O(z) \leq \delta_\O(z)+\|z-w\|^2$ and $ \delta^{1/2}_\O(z)\delta^{1/2}_\O(w) \leq \delta_\O(z)$. Thus, combining \eqref{fa} and \eqref{sa} we see that \eqref{whydoesourestimateleadtobb} holds.
		
		Proposition \ref{profnikolovsclaim} now follows directly from \eqref{uptoadditiveconstant}.\end{proof}
	
	One can see that the estimate \eqref{whydoesourestimateleadtobb} has a real analogue replacing the complex normal direction with the real normal direction.
	
	Let $\O$ be a strongly pseudoconvex domain and recall the well-known estimate \cite{M} for the Kobayashi-Royden metric:
	\begin{equation} \frac{\|v_z\|}{\delta_\O(z)} + \frac{\|v\|}{\sqrt{\delta_\O(z)}} \sim \kappa_\O(z;v), \:\:\:\:\: z \in \O, \: \: v \in \Cn .\label{maestimate} \end{equation}
		
	Recall that by \cite[Proposition 3.1]{P} the Kobayashi-Royden metric is the derivative of the Lempert function.
	
	As $x\sim \log(1+x)$ if $x\geq 0$ is bounded above, the estimate of the Lempert function given in Theorem \ref{profnikolovsdream} implies that for fixed $z\in\O$ and $w\in \O$ sufficiently close to $z$ we have
	$$ l_\O(z,w) \sim {A_\O(z,w)}.$$
	Hence we conclude that if $\O$ has $\mathcal{C}^{2,\alpha}$-smooth boundary, we can get the asymptotic equality \eqref{maestimate} as a consequence of Theorem \ref{profnikolovsdream} in the $\mathcal{C}^{2,\alpha}$-smooth case by fixing $z\in \O$ and letting $w\in \O$ tend to $z$ in the direction of $v$.
	
	By using the estimates we provide in Theorem \ref{profnikolovsdream} for the Bergman distance instead of the Lempert function we recover  (see \cite[Theorem 19.4.6]{JP}) the same estimates for the Bergman metric on strongly pseudoconvex domains with $\mathcal{C}^{3,1}$-smooth boundary:
	
	$$\frac{\|v_z\|}{\delta_\O(z)} + \frac{\|v\|}{\sqrt{\delta_\O(z)}} \sim \beta_\O(z;v), \:\:\:\:\: z \in \O, \: \: v \in \Cn .$$
	
	\section{Preliminaries}\label{transformations}
	
	Let $\Omega$ be a strongly pseudoconvex domain with $\mathcal{C}^{2,\alpha}$-smooth boundary. Some of the estimates considered in the paper are critical when two points $z,w\in\Omega$ approach a point $p\in\partial\Omega$ \emph{complex tangentially}, which throughout the paper means that they converge to a point $p\in\partial\Omega$ with $\|(z-w)_{p}\|/\|z-w\|$ tending to zero (if this is not the case we shall say that they approach the point \emph{complex transversally}). Here, $v_p$ is the complex normal component of the vector $v$ taken with respect to $p\in\partial\Omega$.
	
	To study these cases we will apply a scaling method together with the Fridman-Ma transformations of the domain, as connected together in \cite{KNT}. What is also novel in our approach is the very careful choice of points on geodesics along which we are going to scale as well as quantities that we are going to estimate. Thanks to them we will be able to carry out precise estimates.
	
	For the convenience of the reader, we will briefly recall how to transform a domain and outline the scaling construction we are going to use.
	
	\subsection{Transformation of a domain} Let $\O$ be a strongly pseudoconvex domain with $\mathcal C^{2,\alpha}$-smooth boundary.
	It is well-known that if $x$ is sufficiently close to the boundary, then there exists a unique point, $p:=\pi_\O(x)\in\partial \O$ with $\|p-x\|=\delta_{\O}(x)$.
	
	We will first explain how to modify the domain $\O$ in an appropriate form. The idea is based on the Fridman-Ma \cite{FM} transformation of the domain (see also \cite[Section 2]{KNT} for more details on this construction).
	
	Given $p\in \partial \O$ there exist a biholomorphic mapping $F_p:\overline \O \to \mathbb C^N$ that sends $p$ to $e_1=(1,\ldots, 0)$ and such that $F_p(\O)$ is contained in a ball tangent to $\mathbb B^N$, and near $e_1$ it has a defining function of the form
	\begin{equation}\label{finaldefiningfunction}\rho(z) = -1 +|z|^2 + o(|z-e_1|^{2+\alpha}),\quad \text{as } z\to e_1.
	\end{equation}
	
	It follows from \cite[Lemma 2]{KNT} (see also the proof of \cite[Theorem 4.1]{DGZ}) that for any $z\in \O$ that is sufficiently close to $p$ there exists $q\in \partial \O$, close to $p$, so that $F_q(z)\in (0,1)\times \{0\}^{N-1}.$
	
	Let us denote $F=F_q$. Due to the $\mathcal C^2$-smoothness of the boundary, $\|F(z)-F(w)\|\sim \|z-w\|$ and $\delta_\O(z)\sim \delta_{F(\O)}(F(z))$ for $z,w\in\O$.
	
	Moreover, we will later see in Lemma \ref{letshopethatthisisthefinaldifficulty} that under these transformations we have \begin{align}\label{wds}
		\|(F(z)-F(w))_{F(z)}\|+\|F(z)&-F(w)\|^2+\|F(z)-F(w)\|\delta^{1/2}_{F(\O)}(F(z))\sim 	
		\\\nonumber
		\|(z-w)_z\|+\|z&-w\|^2+\|z-w\|\delta^{1/2}_\O(z) \end{align}
	for $z,w\in\O$ close to $p$ and also that
	\begin{equation}\label{wds1}
		\|(F(z)-F(w))_{F(z)}\| \lesssim  \|(z-w)_z\|+\|z-w\|^2+\|z-w\|\delta_\O(z).
	\end{equation}
	
	Note that as $\O$ has $\mathcal C^{1,\alpha}$-smooth boundary, $F_p$ depends $\mathcal C^\alpha$-continuously on $p$. In particular, this shows that the constants in the estimates \eqref{wds} and \eqref{wds1} change locally uniformly with respect to the base point.

	All of this discussion can be summarized in the following lemma.
	\begin{lem}\label{wemayassume}
		Let $\Omega$ be a strongly pseudoconvex domain in $\mathbb C^N$ with $\mathcal{C}^{2,\alpha}$-smooth boundary and $z\in\overline\O$ be sufficiently close to $\partial \O$. Then, there exists $F\in \mathcal O(\ov \O, \mathbb{C}^N)$ which is a biholomorphism onto its image, sending $z$ to $(1-s,0,...,0)$, which satisfies \eqref{wds} and \eqref{wds1} with constants that are uniform with respect to $z$, and such that a defining function of $F(\O)$ near $e_1$ is of the form \eqref{finaldefiningfunction}.

	\end{lem}

	\subsection{Scaling}	
	Let $z=(z_1,z')$, $m_t(z_1)=\dfrac{z_1+t}{1+t z_1}$ and
	$$ A_t(z)=\left(m_t(z_1),(1-t^2)^{1/2}\dfrac{z'}{1+ t z_1}\right) .$$
	Note that the maps $A^{-1}_t=A_{-t}$ are biholomorphic on $\O$ where $\O$ is as in Lemma \ref{wemayassume}. Set $ A^{-1}_t(\Omega)=\Omega_t$. Recall that (see \cite[Section 2]{KNT}):
	\begin{lem}\label{convergencelemma}
		As $t\rightarrow 1$, $\Omega_t$ converges to the unit ball $\mathbb{B}^N$ in the following sense.
		\begin{enumerate}
			\item For any $\beta>-1$, $\Omega_t\cap \{\re{z_1}>\beta\}$ converges to $\mathbb{B}^N$ in $\mathcal{C}^{2,\alpha}$-sense.
			\item $\Omega_t$ tends to $\mathbb{B}^N$ in Hausdorff sense.
		\end{enumerate}
	\end{lem}
	
	\begin{rem}\label{estimateofcoordinates}
		Let $x,y\in \O$ be close to $p$ and denote $\tilde x = A_t^{-1}(x),$ $\tilde y = A_t^{-1}(y)$. Note that $$A_t(\tilde x) - A_t(\tilde y) = \left( \frac{(1-t^2) (\tilde y^1 - \tilde x^1)}{(1 + t \tilde x^1)(1+t \tilde y^1)}, (1-t^2)^{1/2}\left( \frac{\tilde y'}{1 + t \tilde y^1} - \frac{\tilde x' }{1 + t\tilde x^1}\right) \right),
		$$ where $\tilde x=(\tilde x^1, \tilde x')$, $\tilde y=(\tilde y^1,\tilde y')\in \mathbb C^N$.
		
		If $\re \tilde x_1$ and $\re \tilde y_1$ are far from $-1$, e.g. they are bigger than $-1/2$, the following estimates hold
		\begin{equation}
			\|x_1 - y_1\| \sim (1-t^2) \|\tilde x_1 - \tilde y_1\| \:\:\: \text{and}\end{equation} $$ 	\|x'-y'\| \sim {(1-t^2)^{1/2}} {\|\tilde x' - \tilde y'\|} \:\:\: \text{if additionally} \:\:\:  \|\tilde x'-  \tilde y'\|\gtrsim\|\tilde x_1-\tilde y_1\| .$$
	\end{rem}

	\subsection{Complex geodesics}
	The behavior of complex geodesics in strongly pseudoconvex domains is well understood.
	
	\begin{lem}\label{complexgeodesics} \cite[Lemma 3]{KNT}
		Let $\Omega$ be a strongly pseudoconvex domain with $\mathcal{C}^{2,\alpha}$-smooth boundary.
		\begin{enumerate}
			\item Let $\epsilon'>0$, there exists a small enough $\epsilon>0$ such that if $z,w\in\Omega$ are points satisfying $\|z-w\|\leq \epsilon$, $\delta_\Omega(z)\leq \epsilon$, $\delta_\Omega(w)\leq \epsilon$,  $\frac{\|(z-w)_z\|}{\|z-w\|}\leq \epsilon$ then $z,w$ lie on the image of a complex geodesic $\varphi_{z,w}$ whose Euclidean diameter is smaller than $\epsilon'$.
			\item For $\epsilon>0$ small enough, if $z,w\in\Omega$ are points satisfying $\|z-w\|\leq \epsilon$, $\delta_\Omega(z)\leq \epsilon$, $\delta_\Omega(w)\leq \epsilon$,  $\frac{\|(z-w)_z\|}{\|z-w\|}\leq \epsilon$, then there exists a unique real geodesic joining $z$ to $w$. Furthermore, that real geodesic is induced from a complex one.
			\item There is a constant $C>0$ depending on $\O$ such that for any complex geodesic $\varphi:\Delta\rightarrow\O$ with the Euclidean diameter less than $\epsilon'$, we have $\frac{\|(x-y)_x\|}{\|x-y\|}\leq C \epsilon'$ for any $x\neq y \in \varphi(\Delta)$.
		\end{enumerate}
	\end{lem}
	The first part of the result above is essentially due to \cite[Theorem 1.1]{BFW} (see also \cite[Lemma 4.5]{BST} for an infinitesimal version of this result), the second part is given in \cite[Corollary 4]{KNT}, and the third part follows from \cite[Theorem 3]{KN}.
	
	\begin{rem} It follows from Lemma \ref{complexgeodesics} that if $z$ and $w$ are close to a boundary point $p$ and $z-w$ is close to a complex tangent at $p$, then they are contained in a complex geodesic, say $\varphi$, whose diameter must be small. In particular, after the transformations that were outlined in Lemma~\ref{wemayassume}, we may assume that the range of $\varphi$ is contained in $\O\cap U$, where $U$ is a sufficiently small neighborhood of $p$ such that $\O\cap U$ is strictly convex. This implies that  $\varphi$ is also complex geodesics in $\O\cap U$ and that $k_\O(z,w)= k_{\O\cap U}(z,w)$. In particular, we can confine our study to the case of strictly convex domains. Moreover, according to \cite[Remark 2]{KN}, for a suitably chosen $t\in(0,1)$ an analytic disc $A^{-1}_{t}(\varphi)$ is $\mathcal C^{1,\alpha}$-close to a complex geodesic of the ball that is contained in $\{0\}\times\mathbb{C}^{N-1}$. This reasoning allows us to effectively understand the behavior of $\varphi$.
	\end{rem}
	
	\begin{lemma}\label{lem:tan}
		Let $\O$ be a strongly pseudoconvex domain with $\mathcal C^{2,\alpha}$-smooth boundary and $\varphi:\Delta\rightarrow\Omega$ be a complex geodesic with a sufficiently small diameter. Take $x\in\varphi(\Delta)$ such that $\delta_\O(x)=\max_{\zeta \in \Delta}\delta_\O(\varphi(\zeta))=\|x-p\|$, where $p=\pi_\O(x)$ and reparametrize the complex geodesic so that $\varphi(0) = x$.  Then $\varphi'(0)\in T_p^{\mathbb C}(\partial \Omega)$.
	\end{lemma}
	
	\begin{proof}
		By a rotation and translation we may assume that $x=(-s,0,\ldots ,0)$ and $\pi(x):=\pi_\O(x)=(0,\ldots,0)$. It is well-known that the map $\pi_\O$ is $\mathcal{C}^1$-smooth. Furthermore we have $\pi(\varphi(0))=(0,\ldots ,0)$ and $\re(\pi_1(z))=\mathcal{O}(\|z\|^2)$, where $\pi_1$ denotes the first component of $\pi$.
		
		To prove the assertion we need to show that $\varphi'(0)\in \{0\} \times \mathbb{C}^{N-1}.$ We proceed by a contradiction. Let $\varphi'(0):=\eta=(\eta_1,...,\eta_N)$ with $\eta_1\neq 0$. Then $\varphi(\lambda)=(-s,0,...,0)+\lambda\eta+\mathcal{O}(\lambda^2)$. Composing with a rotation of the disc, we may assume that $\eta_1<0$. Then for $t\in\Delta$ with $t>0$ we see that 	
		$$\delta_{\Omega}(\varphi(t))=\|\varphi(t)-\pi(\varphi(t))\|\geq \re(\varphi(t)_1)-\re(\pi(\varphi(t))_1)\geq s-\eta_1 t + \mathcal{O}(t^2) > s = \delta_\Omega(x) $$ which contradicts the extremality of $x\in\varphi(\Delta)$.
	\end{proof}

	\section{Proofs of Propositions 3 and 4}\label{sectiongeometricproofs}
	
	\begin{proof}[Proof of Proposition~\ref{quantitativevisibility}]

		The key role in our approach is played by a Gehring-Hayman type inequality for strongly pseudoconvex domains:
		\begin{lem}\label{kosinskilemma}\cite[Theorem 1]{KNT}
			Let $\O$ be a strongly pseudoconvex domain with  $\mathcal{C}^{2,\alpha}$-smooth boundary. Then, for any $z,w\in\Omega$ and any Kobayashi geodesic $\gamma_{z,w}:I \to \O$ joining $z$ to $w$, with constants depending only on $\O$ we have the following length estimate
			\begin{equation}
				l(\gamma_{z,w}) \sim \|z-w\|.
			\end{equation}
		\end{lem}
		Therefore, in order to prove Proposition \ref{quantitativevisibility} it is enough to show that there exists a constant $c>0$ such that for any $z,w\in \Omega$ and a Kobayashi geodesic $\gamma_{z,w}$ joining $z$ to $w$, the inequality
		\begin{equation}\label{wtp}  D^\frac{1}{2}_\O(\gamma_{z,w})\geq c\|z-w\| \end{equation}
		holds.
		
		It is well known that any strongly pseudoconvex domain satisfies visibility property with respect to Kobayashi geodesics. Hence \eqref{wtp} requires a proof only if points $z$ and $w$ are close to the same point in the boundary. Take sequences $z_n,w_n$ converging to $p\in\partial \Omega$. We need to show that \eqref{wtp} holds for $z_n$, $w_n$ if $n$ is sufficiently big.
		
		We shall split the proof of Proposition \ref{quantitativevisibility} into two different cases, depending on whether $z_n,w_n$ approach $p$ complex transversally or tangentially.
		
		\smallskip
		
		\textit{Case I.} $z_n,w_n$ approach $p\in\partial \Omega$ complex transversally, i.e. there exists $\alpha>0$ such that
		\begin{equation}\label{eq:trans} \dfrac{\|(z_n-w_n)_p\|}{\|z_n-w_n\|}\geq \alpha .
		\end{equation}
		
		We shall construct particular quasi-geodesics joining $z_n$ with $w_n$. Recall here that a quasi-geodesics is a quasi-isometric embedding of a interval into a metric space and need not be continuous.
		
		Let $\hat{z}_n=\pi_\O(z_n),\hat{w}_n=\pi_\O(w_n)$ be the unique points on $\partial\Omega$ that are closest to $z_n,w_n$, respectively (look at tails of sequences, if necessary). Let $\eta_{z_n},\eta_{w_n}$ denote the inner unit normal vectors to $\partial\Omega$ at $\hat{z}_n,\hat{w}_n$ respectively. Let, moreover, $\sigma_{z_n}:[0,t_n]\to\Omega$ and $\sigma_{w_n}:[0,t_n)\to\Omega$ denote the curves defined by $\sigma_{z_n}(t)=z_n+ t \eta_{z_n},$ $\sigma_{w_n}(t)=w_n+ t \eta_{w_n}$, where $t_n:=\alpha\| z_n - w_n \|/10$.
		

We will first show that $\sigma_{z_n}$ and $\sigma_{w_n}$ are $(2, \epsilon)$-geodesics for some $\epsilon > 0$. Without loss of generality, we will show it for $\sigma_{z_n}$.
		
		As $\O$ is strongly pseudoconvex \begin{equation}\label{strpscestimate}
			k_\O(z,w)\geq \dfrac{1}{2}\left|\log\left(\dfrac{\delta_\Omega(z)}{\delta_\Omega(w)}\right)\right|-c, \:\:\: \: \: z,w\in\O,
		\end{equation}
		for some uniform $c>0$ (see also \cite[Theorem 20]{NO} for an extension of this estimate).
		Moreover, \begin{equation}\label{infinitesimalestimate}\kappa_\Omega(z;v) \leq \delta^{-1}_\Omega(z;v), \:\:\:\:\: z\in \O, \: v\in\Cn, \: v\neq 0. \end{equation}
		It then follows from \eqref{strpscestimate} and \eqref{infinitesimalestimate} that
		$$ \lk_\O(\sigma_{z_n}|_{[\tau_1,\tau_2]})=\int_{\tau_1}^{\tau_2}\kappa_\Omega(\sigma_{z_n}(t);\sigma'_{z_n}(t))dt = \log\left(\dfrac{\delta_\Omega(\sigma_{z_n}(\tau_2))}{\delta_\Omega(\sigma_{z_n}(\tau_1))}\right) \leq 2 k_\O(\sigma_{z_n}(\tau_1),\sigma_{z_n}(\tau_2))+2c $$ so $\sigma_{z_n}$ is actually a $(2,2c)$-geodesic.
		
		We will now show that $\sigma_n:=\sigma_{z_n}\cup\sigma_{w_n}$ is a quasi-geodesic. As we recalled before, a quasi-geodesics need not be continuous. Here we parametrize the curves so that $\sigma_n$ is defined on $I_n:=[0,t_n]\cup (t_n, 2t_n]$ begins at $z_n$ and ends at $w_n$, more explicitly let $\sigma_n(t):=\sigma_{z_n}(t)$ if $t\in[0,t_n]$ and $\sigma_n(t):=\sigma_{w_n}(2t_n-t)$ if $t\in(t_n,2t_n]$. 
		
		Let us consider two points on the curve $\sigma_n$ that are of the form $\sigma_{z_n}(\tau_1),\sigma_{w_n}(\tau_2)$. Integrating as above we get
		\begin{multline}\label{upperboundtolength}
			\lk_\O(\sigma_n|_{[\tau_1,\tau_2]})=\lk_\O(\sigma_{z_n}|_{[\tau_1,t_n]})+\lk_\O(\sigma_{w_n}|_{(t_n, \tau_2]})\leq \\\log\left(1+\dfrac{\|z_n-w_n\|}{\delta_\Omega(\sigma_{z_n}(\tau_1))}\right) + \log\left(1+\dfrac{\|z_n-w_n\|}{\delta_\Omega(\sigma_{w_n}(\tau_2))}\right) .
		\end{multline}
		It follows from the conditions put on the curves, the definition of $t_n$, and \eqref{eq:trans}, that $$\frac{\|(\sigma_{z_n}(\tau_1)-\sigma_{w_n}(\tau_2))_p\|}{\|\sigma_{z_n}(\tau_1)-\sigma_{w_n}(\tau_2)\|}\geq \alpha'$$ for some uniform $\alpha'>0$. By \cite[Proposition 2]{KN} we deduce that
		\begin{multline}\label{lowerboundtolength}
			k_\O(\sigma_{z_n}(\tau_1),\sigma_{w_n}(\tau_2)) \geq \log\left(1+\dfrac{c\|(\sigma_{z_n}(\tau_1)-\sigma_{w_n}(\tau_2))_p\|}{\delta^{1/2}_\Omega(\sigma_{z_n}(\tau_1))\delta^{1/2}_\Omega(\sigma_{w_n}(\tau_2))}\right) \geq \\\log\left(1+\dfrac{c'\|z_n-w_n\|}{\delta^{1/2}_\Omega(\sigma_{z_n}(\tau_1))\delta^{1/2}_\Omega(\sigma_{w_n}(\tau_2))}\right).	
		\end{multline}
		The last inequality above follows from
		$ \|(\sigma_{z_n}(\tau_1)-\sigma_{w_n}(\tau_2))_p\| \gtrsim \|(z_n-w_n)_p\| \gtrsim \|z_n-w_n\| .$
		Then \eqref{upperboundtolength} in combination with \eqref{lowerboundtolength} provides us with $\epsilon > 0$ such that	
		$$ \lk_\O(\sigma_{n}|_{[\tau_1,\tau_2]}) \leq 2 k_\O(\sigma_{z_n}(\tau_1),\sigma_{w_n}(\tau_2))) + \epsilon .$$
		This shows that $\sigma_{n}$ are $(2,\epsilon')$-geodesics, where $\epsilon'=\max\{\epsilon,2c\}$. Note also that $D_\O(\sigma_{n}) \geq \alpha\|z_n-w_n\|/10.$
		
		As $\Omega$ is strongly pseudoconvex, it is Gromov hyperbolic with respect to the Kobayashi distance (see \cite[Theorem 1.4]{BB} for the original argument or combine \cite[Theorem 1.1]{Z} and \cite[Theorem 1.2]{BGNT} for an alternate approach). By the geodesic stability lemma \cite[Theorem 1.7, page 401]{BH} one can deduce that the curves $\sigma_{n}$ and geodesics $\gamma_{z_n,w_n}$ are close to each other in the Hausdorff topology. That is to say, it holds for all $n$ that there exists a constant $H>0$ with \begin{equation}\label{gsl} \max\{k_\Omega(\sigma_n(t), \gamma_{z_n,w_n}(I'_n)),k_\Omega(\gamma_{z_n,w_n}(t'),\sigma_{n}(I_n)):t\in I_n,t'\in I'_n\} \leq H.\end{equation} Above  $I_n'$ denotes the interval $\gamma_{z_n, w_n}$ as defined on and $k_\O (z, A):=\inf k_\O (z,a)$, where the infimum is taken over all $a\in A$.
		Let $x_n$ be a point in $\sigma_n(I_n)$ such that $\delta_\Omega(x_n):=D_\Omega(\sigma_n)$. It follows from \eqref{gsl} that there exists $y_n\in \gamma_{z_n,w_n}(I'_n)$ such that $k_\O(x_n,y_n)\leq H$. By \eqref{strpscestimate}
		$$ \dfrac{1}{2}\log\left(\dfrac{\delta_\Omega(x_n)}{\delta_\Omega(y_n)}\right) \leq H+c.$$
		Consequently
		$$ D_\O(\gamma_{z_n,w_n}) \geq \delta_\Omega(y_n) \geq C \delta_\Omega(x_n) \geq C' \|z_n-w_n\|,$$
		where the constant $C'>0$ is uniform. This shows that \eqref{wtp} holds for the sequences $z_n,w_n$. As those were arbitrary sequences, by a compactness argument we are done. 	
		
		\smallskip
		
		\textit{Case II. } Points $z_n,w_n$ tend to $p\in\partial \Omega$ tangentially, i.e. $$\dfrac{\|(z_n-w_n)_p\|}{\|z_n-w_n\|}\rightarrow 0.$$
		
		We start the proof of this case by noting that the first part of the Lemma \ref{complexgeodesics} implies that provided that $n$ is large enough both $z_n$ and $w_n$ are contained in a complex geodesic tending to $p$ (precisely, $\varphi_n(\overline \Delta)\rightarrow p\in\partial \Omega$ in the Hausdorff topology), which we shall denote by $\varphi_n$. Choose $x_n \in\varphi_n(\Delta)$ such that $\delta_\O(x_n)= \max_{\zeta \in \Delta}\delta_{\O}(\varphi_n(\zeta))=:s_n$ and parametrize $\varphi_n$ so that $\varphi_n(0)=x_n$. Clearly, $s_n$ tends to zero since the diameters of $\varphi_n$ do. Moreover, the derivatives $\varphi_n'(0)$ satisfy the assertion of Lemma~\ref{lem:tan}. The transformation from Lemma \ref{wemayassume} allows us to assume that $p=(1,0,\ldots,0)$ and $x_n=(1-s_n,0,\ldots,0)$. We use the same letter to denote push-forwards of $\varphi_n$ by this transformation. After the transformation the assertion of Lemma~\ref{lem:tan} of $\varphi_n'(0)$ is disturbed, however, thanks to \eqref{wds1} we can still control it:
		\begin{equation}\label{initialgeodesicisvertical}
			|(\varphi_n'(0))_1| \lesssim s_n \|\varphi'_n(0) \|. \end{equation}
		
		The following claim is about the complex geodesics $\varphi_n$ and will be also used in the sequel:
		\begin{claim}\label{claimabouttangentialconvergence}
			Passing to a subsequence, there exists $t_n\in(0,1)$ with $1-t_n=s_n+o(s_n)$ such that the complex geodesics $A^{-1}_{t_n}\circ \varphi_n$ tend in $\mathcal{C}^{1,\alpha}$-topology to a complex geodesic of the unit ball, contained in $\{0\}\times \mathbb{C}^{N-1}$.  Furthermore, $A^{-1}_{t_n}(x_n)$ tends to the origin.
		\end{claim}
		
		\begin{proof}[Proof of Claim \ref{claimabouttangentialconvergence}]
			Choose $t_n\in[0,1-s_n]$ to be the smallest number so that the range $A^{-1}_{t_n}\circ \varphi_n(\overline \Delta)$ intersects $\{\re z_1 = 0 \}$ and set $$\tilde \varphi_n:=A^{-1}_{t_n}\circ\varphi_n.$$
						
			As the diameters of $\varphi_n$ tend to zero, the sequence $(t_n)$ converges to $1$. Then, by Lemma \ref{convergencelemma} and \cite[Remark 2]{KN}, a subsequence of $(\tilde \varphi_n)$ tends to a complex geodesic of the ball. Recall that (see for instance \cite[Example 16.1.1, page 592]{JP}) images of complex geodesics of the unit ball $\mathbb{B}^N$ are the intersection of $\mathbb B^N$ with complex affine lines. The estimate in \eqref{initialgeodesicisvertical} together with elementary calculations shows that for any $t\in[0,1-s_n]$ one has
			$ {\|(A^{-1}_t\circ\varphi'_n(0))_1\|}/{\|(A^{-1}_t\circ\varphi'_n(0))'\|} \lesssim s^{1/2}_n.$ In particular, $[(A^{-1}_{t_n}\circ \varphi_n)'(0)]$ tends in the projective space $\mathbb P(\mathbb C^N)$ to a point of the form $[0:z_2:\ldots:z_N]$.
			The range of $\tilde \varphi_n$ intersects $\{\re z_1 = 0\}$, so $\tilde{\varphi}_n$ tends to a geodesic in $\mathbb B^N$ that is contained in $\{\re z_1 = 0\}\times \mathbb{C}^{N-1}$. As the real part of the first coordinate of $A^{-1}_{t_n}(x_n)$ tends to zero, easy computations give that $1-t_n = s_n + o(s_n)$, whence $\|A^{-1}_{t_n}(x_n)\|=\|\tilde \varphi_n(0)\| = o(s_n)$. This gives the assertion.
		\end{proof}
		
		Let us denote the limit of geodesics that appears in Claim \ref{claimabouttangentialconvergence} by $\psi$.
		
		\begin{rem}\label{remarkaboutdilationofdistances}
			Observe that Remark~\ref{estimateofcoordinates}, the third part of Lemma \ref{complexgeodesics}, and the $\mathcal{C}^{1,\alpha}$-convergence of $\tilde{\varphi}_n$ to $\psi$ show the following
			$$
			\|\tilde z_n - \tilde w_n\|\sim \| \tilde z_n'-\tilde w_n'\| \sim  \dfrac{\|z_n'-w_n'\|}{(1-t^2_n)^{1/2}} \sim \dfrac{\|z_n-w_n\|}{(1-t^2_n)^{1/2}} .$$
		\end{rem}
		
		\medskip
		
		We are coming back to the proof of Case 2. The second part of Lemma \ref{complexgeodesics} implies that there exists a unique real geodesic $\gamma_n:I_n\rightarrow\O$ joining $z_n$ to $w_n$ and its image is contained in the image of $\varphi_n$. Set $\tilde \gamma_n=A^{-1}_{t_n}\circ \gamma_n$.
		
		Recall that the real geodesics of the unit disc $\Delta$ are arcs of circles perpendicular to the unit circle, thus the maximum boundary distance of a real geodesic $\gamma$ on $\psi$ is larger than the Euclidean distance of its endpoints up to a multiplicative constant. In particular, when the complex geodesic intersects a fixed compact set of $\mathbb B^N$, for any real geodesic $\gamma$ joining $x$ with $y$ there exists $u$ lying in the range of $\gamma$ satisfying $\delta_{\mathbb{B}^N}(u)\gtrsim \|x-y\|.$ Then, thanks to Claim \ref{claimabouttangentialconvergence}, $\mathcal C^{1,\alpha}$-convergence of $\tilde\varphi_n$ to $\psi$ implies that for large $n$ there exists $ \tilde u_n\in\tilde\gamma_n(I_n)$ such that
		$ \delta_{\Omega_n}(u_n)\gtrsim \|\tilde z_n-\tilde w_n\|$. In particular, for any $\tilde{q}\in\partial \Omega_{n}$ one has \begin{equation}\label{allptsinbdry}
			\|\tilde{u}_n-\tilde{q}\|\gtrsim \|\tilde z_n - \tilde w_n\|.\end{equation} 	
		Let $q_n:=A_{t_n}(\tilde{q}_n)$ be the closest point on $\partial \Omega$ to $u_n:=A_{t_n}(\tilde{u}_n)$ .
		
		We claim that
		\begin{equation}\label{lowerbound}
			D_\O(\gamma_n) \gtrsim (1-t_n^2)\|\tilde z_n - \tilde w_n\|.\end{equation}
		To see this we distinguish two cases (passing to a subsequence, if necessary). The first one is that $\|\tilde u_n -\tilde q_n \| \sim \|\tilde u_n'- \tilde q_n'\|\gtrsim\|\tilde u_n^1-\tilde q_n^1\|$. Here and in the sequel we use superscripts to denote the coordinates, while subscripts are left for indices. It follows from \eqref{allptsinbdry} and Remark~\ref{estimateofcoordinates} that \begin{multline*}
			D_\O(\gamma_n)=\max\{\delta_\Omega(\gamma_n(t)):t\in I_n\}  \geq \delta_\O(u_n) = \|u_n-q_n\| \gtrsim \|u_n'-q_n'\| \gtrsim \\
			{(1-t^2_n)^{1/2}} \|\tilde{u}'_n-\tilde{q}_n'\| \gtrsim   (1-t_n^2)\|\tilde{u}_n-\tilde{q}_n\| \gtrsim (1-t_n^2)\|\tilde z_n - \tilde w_n\|. \end{multline*}
		
		If the second case holds, i.e. $\|\tilde u_n -\tilde q_n \| \sim \|\tilde u_n^1-\tilde q_n^1\|\gtrsim\|\tilde u_n'- \tilde q_n'\|$, then it again follows from \eqref{allptsinbdry} and Remark~\ref{estimateofcoordinates} that \begin{multline*}
			D_\O(\gamma_n)=\max\{\delta_\Omega(\gamma_n(t)):t\in I_n\} = \|u_n-q_n\| \gtrsim \|u_n^1-q_n^1\| \gtrsim \\
			(1-t_n^2) \|\tilde{u}^1_n-\tilde{q}^1_n\| \gtrsim   (1-t_n^2)\|\tilde{u}_n-\tilde{q}_n\| \gtrsim (1-t_n^2)\|\tilde z_n - \tilde w_n\|. \end{multline*}

		Now \eqref{lowerbound} implies that $D_\O(\gamma_n) \gtrsim (1-t_n^2)\| \tilde z_n - \tilde w_n\|^2$ and this, together with Remark \ref{remarkaboutdilationofdistances}, implies that \eqref{wtp} holds for the sequences $z_n,w_n$. As those were arbitrary, Proposition \ref{quantitativevisibility} follows. \end{proof}
	
	\begin{proof}[Proof of Proposition \ref{themissingpart}]
		Take $\varphi:\Delta\rightarrow \O$ as in the assumptions. We will divide the proof into two different cases.
		
		\textit{Case I.} The diameter of $\varphi$ is bounded below.
		
		If $d_e(\varphi)$ is bounded below, due to the visibility property of strongly pseudoconvex domains (see for instance \cite[Theorem 1.4]{BZ}), it follows that $\varphi(\Delta)$ intersects a fixed compact set. In this case, the first asymptotic equality in~\eqref{nottrivial} follows \cite[Proposition 4]{CHL} and the second asymptotic equality there is a consequence of \cite[Proposition 1]{H}.
		
		\textit{Case II.} The diameter of $\varphi$ tends to zero.
		
		To see the first asymptotic equality in \eqref{nottrivial} when the diameter of $\varphi$ is small (and consequently by Lemma \ref{complexgeodesics}, $\varphi$ is close to a boundary point and in tangential position), one may proceed as in the proof of \cite[Proposition 7]{KN}, by applying the \textit{tangential case} of Proposition \ref{quantitativevisibility} instead of \cite[Theorem 8]{NO}.
		
		Let us focus on the second asymptotic equality. It is clear that $ \max_{z\in\Delta}\|\varphi'(z)\| \gtrsim  d_e(\varphi(\Delta)).$ Thus, what matters is the following estimate \begin{equation} d_e(\varphi (\Delta)) \gtrsim \max_{z\in\Delta}\|\varphi'(z)\| .\label{finalstepforprofnikolovsdream} \end{equation}

		In order to get a contradiction, we suppose that the uniformity in estimate \eqref{finalstepforprofnikolovsdream} fails for a sequence of complex geodesics $\varphi_n:\Delta\rightarrow\O$ such that $d_e(\varphi_n)$ tends to zero. By passing to a subsequence if necessary we may assume that the geodesics $\varphi_n$ tend to $p\in\partial\Omega$. Note that Lemma \ref{complexgeodesics} implies that there is a sequence $\epsilon_n>0$ converging to $0$, such that any two different points $z,w\in \varphi_n(\Delta)$ satisfy the estimate ${\|(z-w)_{z}\|}/{\|z-w\|} \leq \epsilon_n$.
		
		We shall proceed as in the proof of Proposition~\ref{quantitativevisibility}, Case 2. First, we pick points $x_n$ on the ranges of $\varphi_n$ such that $\delta_\O(x_n)= \max_{\zeta\in\Delta} \delta_\O(\varphi_n(\zeta))$, and parametrize $\varphi_n$ so that $\varphi_n(0)=x_n$. By using the transformations from Lemma~\ref{wemayassume} we assume that $\O$ is as in there, $x_n=(1-s_n,0,...,0)=\varphi_n(0)$ and $\delta_\O(x_n)=s_n$. By Claim \ref{claimabouttangentialconvergence} we can take automorphisms of the ball $A^{-1}_{t_n}$ with $1-t_n \sim s_n$ such that $\tilde{\varphi}_n:=A^{-1}_{t_n}\circ\varphi_n$ is a complex geodesic of $\O_n$ uniformly $C^{1,\alpha}$-close to a complex geodesic of the unit ball whose image lies in $\{0\}\times \mathbb{C}^{N-1}$, denoted by $\psi$. Note that for such a geodesic $ d_e(\psi(\Delta))=2\max_{z\in\Delta}\|\psi'(z)\| .$
		As each $\tilde{\varphi}_n$ is $\mathcal{C}^1$-close to a $\psi$, same holds for $\tilde \varphi_n$, namely
		\begin{equation}\label{diameterestimateholdsinimage}
			d_e(\tilde{\varphi}_n) \sim \max_{z\in\Delta}\|\tilde{\varphi}_n'(z)\|.
		\end{equation}
		Note that above we use the fact that $\tilde\varphi_n(0)$ remains compactly in the image of $\tilde \varphi_n$ which follows from Claim \ref{claimabouttangentialconvergence}.
		
		Let $z_n,w_n\in \varphi_n(\Delta)$ be arbitrary and set $\tilde{z}_n:=A^{-1}_{t_n}(z_n),\tilde{w}_n:=A^{-1}_{t_n}(w_n)$.  It follows from Remark \ref{remarkaboutdilationofdistances} that
		$$
		\|z_n-w_n\| \sim (1-t^2_n)^{1/2}\|\tilde{z}_n-\tilde{w}_n\|.
		$$
		In particular, \begin{equation}\label{thefirstbigproblemweovercame} d_e(\varphi_n(\Delta)) \sim (1-t^2_n)^{1/2} d_e(\tilde{\varphi}_n(\Delta)),\:\:\:\:\:\text{and}\:\:\:\:\: \max_{\zeta \in \Delta} \|\varphi'_n(\zeta)\| \sim (1-t^2_n)^{1/2} \max_{\zeta\in \Delta}\|\tilde \varphi'_n (\zeta)\|   .\end{equation}
		Combining \eqref{diameterestimateholdsinimage} and \eqref{thefirstbigproblemweovercame} one infers that $$
		d_e (\varphi_n(\Delta)) \sim {(1-t^2_n)^{1/2}} d_e(\tilde{\varphi}_n(\Delta))  \sim {(1-t^2_n)^{1/2}} \max_{z\in\Delta} \|\tilde{\varphi}_n'(z)\| \sim \max_{z\in\Delta}\|\varphi_n'(z)\|.
		$$
		In particular, \eqref{finalstepforprofnikolovsdream} holds for the sequence $\varphi_n$. As it was arbitrary, the proof of Theorem \ref{themissingpart} follows. \end{proof}
	
	\section{Proof of Theorem 5}\label{sec:geometricproofs}
	
	\begin{proof}[Proof of Theorem \ref{profnikolovsconjecture}] Within this part we shall use the scaling argument, as well. However, we need to deal with the intricate term $(z-w)_z$ which makes the proof more subtle. We will first show that the estimates we want to achieve are preserved under the transformations we give in Section~\ref{transformations}, that is we will show that \eqref{wds} holds. In fact, we will prove the following:
		
		\begin{lemma}\label{letshopethatthisisthefinaldifficulty}
			Let $\O_1,\O_2\subset\Cn$ be bounded domains with $\mathcal C^2$-smooth boundaries and $f:\ov\O_1\rightarrow\ov\O_2$ be a biholomorphism. For $z,w\in\O_1$ one has
			\begin{equation}\label{e1}\|(f(z)-f(w))_{f(z)}\|\lesssim \|(z-w)_z\|+\|z-w\|^2+ \delta_{\O_1}(z) \|z-w\|
			\end{equation}
			and
			\begin{multline}\label{e2}\|(z-w)_z\|+\|z-w\|^2+ \delta^{1/2}_{\O_1}(z)\|z-w\| \sim \\ \|(f(z)-f(w))_{f(z)}\|+\|f(z)-f(w)\|^2+ \delta^{1/2}_{\O_2}(f(z))\|f(z)-f(w)\|. \end{multline}
		\end{lemma}
		
		\begin{proof}[Proof of Lemma \ref{letshopethatthisisthefinaldifficulty}]
			Let us prove \eqref{e1}. By isometric transformations we may assume that $z=(-s,0,\ldots,0)$ and $\pi_\O(z)=(0,\ldots,0)$. Post-composing
			with $(df |_0)^{-1}$, we may assume $df |_0$ is identity, and
			we have
			$$ f(z)=(z_1 + P(z), z' + Q(z)) + \vec{o}(\|z\|^2), $$ where $P,Q$ are quadratic polynomials and $\vec{o}$ means that the quantity we consider is a vector in $\Cn$. We may write $f(z)-f(w)= (z_1-w_1+P(z)-P(w), w' + Q(z)-Q(w)) + \vec{o}(|z_1-w_1|+\|w'\|^2).$ Since $|z_1-w_1|=\|(z-w)_z\|$ and $\|w'\|^2\leq \|z-w\|^2$, we are able to estimate the $\vec{o}(|z_1-w_1|+\|w'\|^2)$ term.
			
			Let $\tilde \delta _\O$ be the signed boundary distance function and $g_\O(z):=2 \ov \partial \tilde\delta_\O(z)$. Recall that
			the unit normal to $\partial \O$ at $\pi_\O(z)$ is given by $g_\O(z)=g_\O(\pi_\O(z))$.
			
			Note that $f(z)=(-s+\gamma_1 s^2, \gamma' s^2)$, whence
			$$ \|g_{\O_2}(f(z))-g_{\O_2}(0)\| = \|g_{\O_2}(f(z)) - (1,0,...,0)\| \lesssim s = \delta_{\O_1}(z) .$$ Then, by Cauchy-Schwarz and triangle inequalities, we can write
			\begin{multline*} \|(f(z)-f(w))_{f(z)}\| =
				\langle f(z)-f(w), (1,0,...,0) + \vec{ \mathcal{O}}(\delta_{\O_1}(z)) \rangle  = \\
				|z_1- w_1| + \|P(z)-P(w)\| + (\|w'\|+\|Q(z)-Q(w)\|) \mathcal{O}(\delta_{\O_1}(z)) .
			\end{multline*}
			The first term above is just $\|(z-w)_z\|$. To bound $\|P(z)-P(w)\|$ write $\|P(z)-P(w)\| \leq \|P(z_1,0,...,0) - P(w_1,0,...,0)\| +  \|P(w_1,0,...,0) - P(w)\|$. Observe that $ \|P(z_1,0,...,0) - P(w_1,0,...,0)\| \lesssim |z_1^2-w_1^2| \lesssim |z_1- w_1| = \|(z-w)_z\| .$ To deal with $\|P(w_1,0,...,0) - P(w)\|$ write $P(w_1,0,...,0) - P(w) = w_1 L(w') + S(w')$, where $L$ is linear and $S$ is a quadratic polynomial. Clearly, $|S(w')|\lesssim\|w'\|^2 \leq \|z-w\|^2$. Now, observe that $$|w_1 P(w')|\leq |z_1| |P(w')| + |z_1-w_1||P(w')| \lesssim \delta_{\O_1}(z)\|w'\| + |z_1-w_1| \lesssim \delta_{\O_1}(z)\|z-w\| + \|(z-w)_z\|.$$
			Putting all of these estimates together we infer that
			\begin{multline*} |P(z)-P(w)| \leq |P(z_1,0,...,0) - P(w_1,0,...,0)| +  |P(w_1,0,...,0) - P(w)| \lesssim\\
				\delta_{\O_1}(z)\|z-w\| + \|(z-w)_z\| + \|z-w\|^2.
			\end{multline*}
			
			Similar estimates as above show that $(\|Q(z)-Q(w)\|+\|w'\|)\delta_{\O_1}(z)\lesssim \delta_{\O_1}(z)\|z-w\|.$ Consequently,
			\begin{multline*}\|(f(z)-f(w))_{f(z)}\| \lesssim|z_1- w_1| + \|P(z)-P(w)\| + (\|w'\|+\|Q(z)-Q(w)\|) \delta_{\O_1}(z) \lesssim \\
				\delta_{\O_1}(z)\|z-w\| + \|(z-w)_z\| + \|z-w\|^2.
			\end{multline*}
			
			To show \eqref{e2} note that the distances are preserved, up to multiplicative constants, by $\mathcal C^1$-smooth maps. Therefore is enough to show that the term $\|(f(z)-f(w))_{f(z)}\|$ can be estimated by $\|(z-w)_z\|+\|z-w\|^2+ \delta^{1/2}_{\O_1}(z)\|z-w\|$ (to get the opposite estimate it is enough to apply the above to the inverse of $f$). This follows immediately from \eqref{e1}.
			
		\end{proof}
				
		We come back to the proof of Theorem \ref{profnikolovsconjecture}. The lower bound for the diameter follows from \cite[Theorem 3]{KN} and Proposition \ref{themissingpart}.
		
		We shall prove the upper one. We shall actually give a stronger estimate \begin{equation}\label{28} d_e(\varphi) \lesssim \|(z-w)_z\|/\|z-w\|+\delta^{1/2}_\O(z).
		\end{equation} Parametrize $\varphi$ so that $\delta_\O(\varphi(0))=\max_{\zeta\in\Delta}\delta_\O(\varphi(\zeta))$ and let $z\in\varphi(\Delta)$. We will show that the uniform estimate
		\begin{equation}
			\dfrac{\delta_{\Omega}(x)}{\delta_{\Omega}(z)} \sim \dfrac{1}{\delta_\Delta(\varphi^{-1}(z))} \label{interiorcaseisok}
		\end{equation}
		holds, where $x=\varphi(0)$. To prove it we will rely on the argument used in \cite[page 6]{KN}. By direct computations and \cite[Theorem 20]{NO} we get that
		$$ \dfrac{1}{2}\log\left(\dfrac{c\delta_{\Omega}(x)}{\delta_{\Omega}(z)}\right)\leq k_\O(x,z) = k_\Delta(0,\varphi^{-1}(z)) \leq \dfrac{1}{2} \log\left(\dfrac{2}{\delta_\Delta(\varphi^{-1}(z))}\right).$$ Thus one side of \eqref{interiorcaseisok} is clear. To see the opposite one, observe that, as $\O$ is strongly pseudoconvex, there exists $c>0$ such that
		\begin{multline}\label{ddD} \dfrac{1}{\delta_\Delta(\varphi^{-1}(z))} \sim \dfrac{1}{1-|\varphi^{-1}(z)|^2} = \k_\Delta(\varphi^{-1}(z);1) = \\ \k_\O(z;\varphi'(\varphi^{-1}(z))) \leq \dfrac{c\|\varphi'(\varphi^{-1}(z))\|}{\delta^{1/2}_\O(z)}+ \dfrac{\|(\varphi'(\varphi^{-1}(z)))_z\|}{\delta_\O(z)} .
		\end{multline}
		According to \cite[Proposition 7]{KN}, the estimate in \eqref{ddD} implies that
		$$ \delta_\O(z) \lesssim d_e(\varphi)\|\varphi'(\varphi^{-1}(z))\| \delta_\Delta(\varphi^{-1}(z)).$$
		By Proposition \ref{themissingpart} we get $d_e(\varphi)\|\varphi'(\varphi^{-1}(z))\| \lesssim \delta_\O(x)$. Two last inequalities prove \eqref{interiorcaseisok}.
		
		Observe that \eqref{interiorcaseisok} and Proposition \ref{themissingpart} imply that if $\varphi^{-1}(z)\in \Delta$ is contained in a fixed compact subset, then $$ \delta^{1/2}_\O(z) \sim \delta^{1/2}_\O(x) \sim d_e(\varphi),$$ so in this case the upper bound holds, i.e. $d_e(\varphi)$ can be estimated by $\delta_\O^{1/2}(z)$, whence \eqref{28} follows.
		
		Thus, the upper bound requires the proof only if $\varphi^{-1}(z)$ is close to the unit circle. We shall show that then $d_e (\varphi)$ can be bounded by the term $\| (z-w)_z \|/\|z-w\|$. Precisely, we shall prove the following:
		
		\begin{claim}\label{lefttoshow} For any $c\in (0,1)$ and for any complex geodesic $\varphi:\Delta\rightarrow \O$ parametrized so that the maximal boundary distance is attained at the origin,
			\begin{equation}\label{claim20}
				\|(z-w)_z\| \gtrsim d_e (\varphi) \|z-w\|, \:\:\: \:\: \text{provided that} \:\: \:\:\: |\varphi^{-1}(z)| \ge c.
			\end{equation}
		\end{claim}
		
		\begin{proof}[Proof of Claim~\ref{lefttoshow}] We need to prove the assertion for a sequence of complex geodesics $\varphi_n: \Delta \rightarrow \O$ parametrized so that $\max_{\zeta\in\Delta}\delta_\O (\varphi_n(\zeta)) =\delta_{\Omega}(\varphi_n(0))$ and points $z_n,w_n\in \varphi_n(\Delta)$ such that $|\zeta_n|:=|\varphi^{-1}_n(z_n)| \ge c$.
			Suppose that \eqref{claim20} fails. As complex geodesics passing through a compact set intersect the boundary transversally, $d_e(\varphi_n)$ must tend to zero (use e.g. the visibility property of strongly pseudoconvex domains). Passing to a subsequence, if necessary, we can assume that $\varphi_n$ tends to $p\in\partial\Omega$.
			
			Set $x_n=\varphi_n(0)$ and for each $n$ apply the transformation given in Lemma~\ref{wemayassume}. Then $x_n=(1-s_n,0,...,0)$, where $s_n > 0$ tends to zero. Claim ~\ref{claimabouttangentialconvergence} then shows that passing to a subsequence, if necessary, we can find automorphisms of the ball $A^{-1}_{t_n}$, where $A^{-1}_{t_n}\circ\varphi_n=\tilde{\varphi}_n$ are uniformly $\mathcal C^{1,\alpha}$-close to a complex geodesic of the unit ball that is contained $\{0\}\times \mathbb{C}^{N-1}$. By Proposition \ref{themissingpart} we have $1-t_n\sim s_n \sim d^2_e (\varphi_n) .$
			
			Rotations in $\{0\}\times \mathbb{C}^{N-1}$ do not affect the estimates we want to achieve, so taking a proper one we may assume that $\tilde \varphi_n$ is uniformly $\mathcal{C}^{1,\alpha}$-close to the geodesic $\psi(\zeta)=(0,\zeta,0,...,0)$, $\zeta\in \Delta$.
			
			Let us write $\tilde\varphi_n(\zeta)=(0,\zeta,0')+(e_n^1(\zeta),e^2_n(\zeta),e^3_n(\zeta))$, where $0', e^3_n(\zeta) \in \mathbb{C}^{N-2}$.  Let $\omega_n=\varphi^{-1}_n(w_n)$. As $\varphi_n=A_{t_n}\circ\tilde \varphi_n$, explicit computations give the formulas:
			\begin{align}\label{coordinatesfromthedisc}
				z_n&=\left(\dfrac{e^1_n(\zeta_n)+t_n}{1+t_n e^1_n(\zeta_n)},\dfrac{(1-t_n^2)^{1/2}(\zeta_n + e^2_n(\zeta_n))}{1+t_n e^1_n(\zeta_n)},\dfrac{(1-t_n^2)^{1/2} e^3_n(\zeta_n)}{1+t_n e^1_n(\zeta_n)}\right),\\ \nonumber
				w_n&=\left(\dfrac{e^1_n(\omega_n)+t_n}{1+t_n e^1_n(\omega_n)},\dfrac{(1-t_n^2)^{1/2}(\omega_n + e^2_n(\omega_n))}{1+t_n e^1_n(\omega_n)},\dfrac{(1-t_n^2)^{1/2} e^3_n(\omega_n)}{1+t_n e^1_n(\omega_n)}\right).
			\end{align}
			
			To get a contradiction it is enough to prove that for large enough $n$ we have
			\begin{equation}\label{claimaboutinnerproduct}
				\|(z_n-w_n)_{z_n}\|= |\langle z_n-w_n,\eta_{z_n}\rangle |\gtrsim d_e (\varphi_n) \|z_n-w_n\|,
			\end{equation}
			where $\eta_{z_n}=g_\O(z_n)= 2 \ov \partial \tilde \delta_\O(z_n)$ is the complex unit normal to $\partial\O$ taken with respect to $\hat z_n:= \pi_\O(z_n) \in \partial\O.$
			Since $\delta_\O(z_n)\eta_{z_n}=\hat z_n - z_n$ and $\O$ has $\mathcal{C}^{2,\alpha}$-smooth boundary it is easy to see that $$  \eta_{z_n}=\hat{z}_n+\vec{\mathcal{O}}(\|\hat z_n-p\|^{1+\alpha})=\dfrac{z_n}{1-\delta_{\Omega}(z_n)}+\vec{\mathcal{O}}(\|\hat z_n-p\|^{1+\alpha}),$$ where $\vec{\mathcal{O}}(.)$ means that the quantity we consider is an element of $\Cn$.
			
			Observe that
			\begin{equation}
				\label{bigohislessthandiameter}
				\|\hat{z}_n - p \| \lesssim d_e (\varphi_n).
			\end{equation}
			Indeed, $	\|\hat{z}_n - p \| \leq 	\|\hat{z}_n - z_n \|+\|z_n-x_n\|+\|x_n-p_n\| \leq 2 \delta_{\Omega}(x_n)+ d_e(\varphi_n(\Delta))$. Therefore, \eqref{bigohislessthandiameter} follows from Proposition~\ref{themissingpart}.
			
			Therefore, \begin{equation}\label{unitnormalequation} \eta_{z_n} =\frac{z_n}{1-\delta_\O(z_n)} + E_n(z_n),\end{equation} where $|E_n(z_n)| \leq d_e(\varphi_n(\Delta)) d_e(\varphi_n(\Delta))^{\alpha}$. Formulas \eqref{coordinatesfromthedisc} and \eqref{unitnormalequation} show that $$ \langle z_n-w_n,\eta_{z_n} \rangle = \dfrac{(z_n^1-w_n^1) \ov z^1_n} {1-\delta_\O(z_n)} + \dfrac{ (z_n^2-w_n^2) \ov z^2_n}{1-\delta_\O(z_n)} +  \dfrac{(z'_n-w'_n) \ov z'_n}{1-\delta_\O(z_n)} + \langle z_n-w_n,E_n(z_n) \rangle .$$
			
			We shall estimate the quantity given above. First note that by Cauchy-Schwarz inequality and \eqref{bigohislessthandiameter} one deduces that
			$$ | \langle z_n-w_n,E_n(z_n) \rangle | \leq \|z_n-w_n\| \|E_n(z_n)\| \lesssim \|z_n-w_n\|  d_e(\varphi_n(\Delta)) d_e(\varphi_n(\Delta))^{\alpha} .$$
			
			Let us set $$A_n=\dfrac{(z_n^1-w_n^1) \ov z^1_n} {1-\delta_\O(z_n)},\ B_n = \dfrac{ (z_n^2-w_n^2) \ov z^2_n}{1-\delta_\O(z_n)},\ C_n = \dfrac{(z'_n-w'_n) \ov z'_n}{1-\delta_\O(z_n)} . $$
			
			We will first estimate $B_n$. Note that $d_e(\varphi_n) \sim s^{1/2}_n \sim (1-t^2_n)^{1/2}$  by Proposition \ref{themissingpart} and Claim \ref{claimabouttangentialconvergence}, whence using \eqref{coordinatesfromthedisc} and Remark \ref{remarkaboutdilationofdistances} we get
			$$ |B_n | = \|(z_n^2-w_n^2) \ov z^2_n\| \sim (1-t_n^2) | \zeta_n | |\zeta_n - \omega_n| \sim d_e (\varphi_n) \|z_n-w_n\| | \zeta_n | .$$
			Similarly, $$ |A_n| \sim (1-t_n^2) \|e^1_n(\zeta_n) - e^1_n (\omega_n) \| .$$
			Note that $\mathcal{C}^{1}$-convergence implies that there are constants $\epsilon_n$ tending to zero such that
			$$ |e^1_n(\zeta_n) - e^1_n (\omega_n) | \leq \epsilon_n |\zeta_n-\omega_n|. $$ Consequently, according to Remark \ref{remarkaboutdilationofdistances},
			$$ | A_n | \lesssim  \epsilon_n (1-t_n^2)^{1/2} \|z_n-w_n\|  \lesssim  \epsilon_n d_e(\varphi_n) \|z_n-w_n\| .$$
			Similarly,
			\begin{multline*} | C_n | \lesssim   (1-t_n^2) \max\{|e^1_n(\zeta_n) - e^1_n (\omega_n) |, |e^3_n(\zeta_n) - e^3_n (\omega_n) |\}|e^3_n(\zeta_n)| \lesssim \\ (1-t_n^2)^{1/2}  \| z_n - w_n \| |e^3_n(\zeta_n)| \lesssim d_e (\varphi_n) | \zeta_n - \omega_n | |e^3_n(\zeta_n)| .
			\end{multline*}
			As a consequence of $\mathcal{C}^{1,\alpha}$-convergence of the complex geodesics, $ | e^3_n(\zeta_n)| $ tends to zero.
			
			Summing up, we have obtained the following.
			
			$$ \langle z_n-w_n,\eta_{z_n} \rangle = A_n + B_n + C_n + \langle z_n-w_n, E_n(z_n) \rangle, $$ where
			$$ \dfrac{| A_n | + | C_n | + |\langle z_n-w_n, E_n(z_n) \rangle|} {d_e (\varphi_n) \|z_n-w_n\|} \rightarrow 0 $$ and
			$$ |B_n| \sim |\zeta_n| d_e(\varphi_n) \|z_n-w_n\| .$$
			Since $|\zeta_n |\ge c$, it follows that $|\langle z_n-w_n,\eta_{z_n} \rangle | \sim B_n \gtrsim d_e (\varphi_n) \|z_n-w_n\| .$ Hence \eqref{claimaboutinnerproduct} holds and a contradiction is derived.
		\end{proof}
		
		Claim~\ref{lefttoshow} is proven, and Theorem \ref{profnikolovsconjecture} follows. \end{proof}

	\section{Proofs of Theorem 1, Proposition 2, and Corollary 8}\label{estimateproofs}
	
	\begin{proof}[Proof of Proposition \ref{normallowerbound}] A partial case of this result is given in \cite[Proposition 2]{KN}. To extend this result to full generality, note that the \textit{tangential case} of Proposition \ref{themissingpart} implies that the equation (6) in \cite{KN} holds. Then the result follows from \cite[Proposition 8]{KN}. We refer the reader to the proofs of \cite[Proposition 8]{KN} and \cite[Proposition 2]{KN} for more details. \end{proof}
	
	\begin{proof}[Proof of Theorem \ref{profnikolovsdream}]
		
		For clarity, we shall divide the proof into several subproofs.
		
		\medskip

		\noindent\textit{Proof of the lower bound for the Kobayashi distance.} Recall the estimate given in \cite[Proposition 1.6, Proposition 1.7]{NT}: if $\O$ is a strongly pseudoconvex domain with $\mathcal{C}^{2}$-smooth boundary, then there exists $c>0$ such that
		\begin{equation}\label{partofthelowerbound} k_\O(z,w)\geq \log\left(1 + \dfrac{c\|z-w\|}{\delta^{1/2}_\O(z)}\right)\left(1 +  \dfrac{c\|z-w\|}{\delta^{1/2}_\O(w)}\right), \:\:\:\:\: z,w \in \O.\end{equation}
		Thus the lower bound follows from Proposition \ref{normallowerbound} together with \eqref{partofthelowerbound}.
		
		\medskip
		
		\noindent\textit{Proof of the upper bound for the Kobayashi distance}. Recall that, by \cite[Corollary 8]{NA}, for a domain $\O$ with Dini-smooth boundary one has the following estimate:	\begin{equation}\label{diniestimate}k_\O(z,w)\leq \log \left(1+\dfrac{C\|z-w\|}{\delta^{1/2}_\O(z)\delta^{1/2}_\O(w)}\right), \:\:\:\:\: z,w\in\O.\end{equation}
			
		By \eqref{diniestimate} we see that if the upper bound in Theorem \ref{profnikolovsdream} fails for sequences $z_n,w_n\in\O$, then passing to subsequences if necessary, we must have $z_n,w_n\rightarrow p\in\partial\O$ and $\frac{\|(z_n-w_n)_p\|}{\|z_n-w_n\|}\rightarrow 0$. We will show that this is impossible.
		
		Looking at a tail of the sequence Lemma~\ref{complexgeodesics} asserts that $z_n,w_n\in \Omega$ lie in the image of a complex geodesic $\varphi_n$ which tends to $p$. Parametrize $\varphi_n$ so that maximal boundary distance is attained at $\varphi_n(0)=:x_n\in\varphi_n(\Delta)$ and assume that $\Omega$ is as in Lemma \ref{wemayassume}. Then $x_n=(1-s_n,0,...,0)$ and $\delta_\O(x_n)=s_n$. Taking an automorphism of the ball $A^{-1}_{t_n}$ as in Claim \ref{claimabouttangentialconvergence} shows that for large $t_n$ the analytic disc $\tilde\varphi_n:=A^{-1}_{t_n}\circ \varphi_n$ is uniformly close to a complex geodesic of the ball contained in $\{0\}\times \mathbb{C}^{N-1}$. Furthermore, we have $1-t_n\sim s_n = \delta_\O(x_n)$. Set $\tilde{z}_n = A^{-1}_{t_n}(z_n)$, $\tilde{w}_n= A^{-1}_{t_n}(w_n)$.
		
		We observed in Remark \ref{remarkaboutdilationofdistances}
		and the proof of Proposition \ref{quantitativevisibility} that the following estimates hold:	$(1-t^2_n)\delta_{\O_n}(\tilde z_n)\lesssim\delta_\O(z_n), $ and $(1-t^2_n)\delta_{\O_n}(\tilde w_n) \lesssim \delta_\O (w_n)  .$
		We claim more, namely that there are asymptotic equalities within them. Precisely, we will show the following: if $y_n$ lies in the range of $\varphi_n$ and $\tilde y_n$ is its pull-back to $\tilde \varphi_n$, then
		\begin{equation}\label{eq:cl} (1-t^2_n)\delta_{\O_n}(\tilde y_n)\sim\delta_\O(y_n).
		\end{equation}
		
		To see this, put $\delta'_{\O}(y_n)=\inf\{\|u-y_n\|:u\in\partial \O, u^1= y^1_n\}, \delta'_{\O_n}(\tilde y_n)=\inf\{\|u-\tilde y_n\|:u\in\partial \O_n, u^1=\tilde y^1_n\}$, and $\delta^1_{\O}( y_n)=\inf\{\|u-y_n\|:u\in\partial \O, u'= y'_n\}$. Observe that, as $\tilde \varphi_n$ are uniformly $\mathcal{C}^{1,\alpha}$-close to a complex geodesic of the ball lying in $\{0\}\times\mathbb{C}^{N-1}$, we have $$ \delta_{\O_n}(\tilde y_n)\sim \delta'_{\O_n}(\tilde y_n) .$$
		Moreover by Remark~\ref{estimateofcoordinates}, $\delta'_\O(y_n)\sim (1-t^2_n)^{1/2} \delta'_{\O_n}(\tilde y_n)$.
		
		To get \eqref{eq:cl} it is enough to show that
		\begin{equation}\label{eq:cl1} \delta_\O(y_n) \lesssim \delta^1_\O(y_n) \lesssim (1-t^2_n)^{1/2}\delta'_\O(y_n).
		\end{equation}
		The first estimate in \eqref{eq:cl1} is trivial, only the second one requires a proof. Since $\tilde y_n$ belongs to the image of $\tilde \varphi_n$, the estimate $|\tilde y^1_n|<\epsilon_n$ holds, where $\epsilon_n$ is a sequence of positive numbers converging to zero. Pushing $\tilde y_n$ forward by $A_{t_n}$ one can write $y_n^1=t_n+  (1-t_n)o_n,$ where $\text{o}_n$ are complex numbers converging to $0$. Let us make a substitution $t_n+ i (1-t_n)\text{o}_n=\tau_n+ i (1-\tau_n)o_n$, where $o_n \in \mathbb{R}$ (to be precise, we define $\tau_n:= t_n + (1-t_n) \re \text{o}_n$ and $o_n:=\im \text{o}_n/(1-\re \text{o}_n)$). It is clear that $o_n$ converges to $0$ and that $1-\tau_n \sim 1-t_n$. Thus we can write $y_n=(\tau_n+i(1-\tau_n)o_n,y'_n)\in (\mathbb R + i\mathbb R)\times \mathbb C^{N-1})$.
		
		Take $\sigma_n > \tau_n$ such that $(\sigma_n+i(1-\tau_n)o_n,y_n')\in\partial \Omega$. Moreover, take $\xi'_n$ such that $a_n:=(\tau_n+i(1-\tau_n) o_n, \xi'_n)$ is in $\partial \O$ and satisfies $\|a_n-y_n\|= \delta'_\O(z_n)$.
		
		Near $p=(1,0,...,0)$ we can write $\partial \O$ as $\re z_1 = u(\text{Im} z_1, z')$, where $$u(\text{Im} z_1, z') =
		1 - \frac{|\text{Im} z_1|^2}{2} - \frac{\|z'\|^2}{2} +
		\mathcal{O}(|\text{Im} z_1|^{2+\alpha} + \|z'\|^{2+\alpha}).$$
		Then, in particular, $\tau_n=u((1-\tau_n)o_n,\xi'_n)$ and  $\sigma_n=u((1-\tau_n)o_n,y'_n)$.
		
		Clearly $\delta^1_\O(y_n)\leq   |\tau_n-\sigma_n| = \| u((1-\tau_n)o_n,\xi'_n)-  u((1-\tau_n)o_n,y'_n) \|$ and $\delta'_\O(y_n) = \|\xi'_n-y'_n\|$. Since $\|\xi'_n\|,\|y'_n\|$ behave as $\mathcal{O}((1-\tau^2_n)^{1/2})$, we get
		$$ \frac{ \delta^1_\O(y_n)}{ \delta'_\O(y_n)} \leq \dfrac{\| u((1-\tau_n)o_n,\xi'_n)-  u((1-\tau_n)o_n,y'_n) \|}{\|\xi'_n-y'_n\|} \lesssim \|\xi'_n\|+\|y'_n\| \lesssim (1-\tau^2_n)^{1/2} \sim (1-t^2_n)^{1/2}.$$
		This yields \eqref{eq:cl1} and consequently leads to \eqref{eq:cl}.

		Summing up, it follows from  what we showed, Remark \ref{remarkaboutdilationofdistances}, and the proof of Proposition \ref{quantitativevisibility} that the following estimates hold:
		$$ \|z_n-w_n\| \sim (1-t^2_n)^{1/2}\|\tilde z_n-\tilde w_n\|, \: \: \: \delta_{\Omega}(z_n)\sim (1-t^2_n) \delta_{\Omega_n}(\tilde{z}_n), \: \: \: \delta_{\Omega}(w_n)\sim (1-t^2_n) \delta_{\Omega_n}(\tilde{w}_n) .$$		
		All of them together with \eqref{diniestimate} imply that 	
		$$ k_\O(z_n,w_n) = k_{\O_n}(\tilde{z}_n,\tilde{w}_n) \leq \log\left(1+\dfrac{C\|\tilde z_n -\tilde w_n\|}{\sqrt{\delta_{\Omega_n}(\tilde z_n)\delta_{\Omega_n}(\tilde w_n)}}\right) \leq \log\left(1+\dfrac{C'\|z_n-w_n\|(1-t^2_n)^{1/2}}{\sqrt{\delta_{\Omega}(z_n)\delta_{\Omega}(w_n)}}\right),$$
		with a uniform $C'>0$.
		Recall that by Claim~\ref{claimabouttangentialconvergence} $(1-t^2_n)^{1/2}\sim s^{1/2}_n = \delta^{1/2}_{\Omega}(x_n)\sim d_e(\varphi_n(\Delta))$. Therefore, the upper bound for the Kobayashi distance follows from Theorem~\ref{profnikolovsconjecture}.
		
		\medskip
		
		\noindent \textit{Extension of estimates to the Lempert function and the Carath{\'e}odory distance}. Recall the following comparison results by \cite[Theorem 1.6]{NT} and \cite[Remark 1.10]{NT} (see also \cite[Theorem 1]{N} for more precise estimates) about invariant functions of strongly pseudoconvex domains:$$l_\O(z,w) \leq c_\O(z,w) + C_0 g_\O(z,w) \:\:\: \:\: \text{and} \:\:\:\:\: l_\O(z,w) \leq C_0 c_\O(z,w),$$ where $C_0>1$ and $$ g_\O(z,w) = \dfrac{\|z-w\|}{\|z-w\|^{1/2}+\delta^{1/2}_\O(z)+\delta^{1/2}_\O(w)} . $$
		
		Observe that $g_\O(z,w)$ is bounded above. In particular, as the Kobayashi distance is bounded below by the Carath{\'e}odory distance and bounded above by the Lempert function, we see that there exists a constant $C_1>1$ such that
		\begin{align}
			k_\O(z,w) &\leq c_\O(z,w) + C_1, \:\: k_\O(z,w)\leq  C_1 c_\O(z,w), \label{comparasion} \\ \nonumber
			l_\O(z,w)  & \leq k_\O(z,w) +  C_1, \:\: l_\O(z,w) \leq  C_1 k_\O(z,w) .
		\end{align}
		
		\medskip

		\noindent\textit{Proof of the upper bound for the Lempert function.} Suppose that the upper bound fails. Then we can find $z_n,w_n\in\O$ tending to $a,b\in\ov\O$, and $c_n\in\mathbb{R}$ tending to infinity such that
		\begin{equation}\label{contradictionforlempert}
			l_\O(z_n,w_n) \geq \log\left(1+{c_n A_\O(z_n,w_n)}\right).
		\end{equation}
		We need to consider two cases (pass to a subsequence, if necessary):	
		
		\textit{Case I.} The sequence $l_\O(z_n,w_n)$ is bounded from above. Then trivially $k_\O(z_n,w_n) \lesssim 1$. In particular, by the lower bound we obtained for the Kobayashi distance, this implies that $A_\O(z_n,w_n)\lesssim 1$. Applying the upper bound for the Kobayashi distance and \eqref{comparasion} we get $$ l_\O(z_n,w_n) \leq C_1 k_\O(z_n,z_n) \leq {C_2 A_\O(z_n,w_n)} .$$
		
		Note that for $r'\geq 0$ one can find a constant $C'>1$ such that $r \leq \log(1+C' r)$ holds for $0\leq r \leq r'$. In light of this, as $A_\O(z_n, w_n)$ are uniformly bounded, the inequality above contradicts \eqref{contradictionforlempert}.
		
		\textit{Case II.} The sequence $l_\O(z_n,w_n)$ is bounded away from $0$. It follows from \eqref{comparasion} and the lower bound we obtained for the Kobayashi distance that $k_\O(z_n,w_n)$, and consequently $ A_\O(z_n,w_n)$, are bounded away from $0$. Using \eqref{comparasion} we get
		\begin{multline*}
			l_\O(z_n,w_n) \leq k_\O(z_n,w_n) + C_3 \leq \log\left(1+ {C A_\O(z_n,w_n)}\right) + C_3 \leq \\ \log\left({C_4 A_\O(z_n,w_n)}\right) + C_3  \leq \log\left(1+ {C_5 A_\O(z_n,w_n)}\right),
		\end{multline*}
		where the constant $C_4$ depends on $\epsilon$ such that $A_\O(z_n,w_n)\geq \epsilon$.
		
		Having this estimate, we once again see that a sequence fulfilling \eqref{contradictionforlempert} cannot exist.
		
		\medskip

		\noindent\textit{Proof of the lower bound for the Carath{\'e}odory distance.} Suppose that this lower bound fails. We can then find $z_n,w_n\in\O$ tending to $a,b\in\ov\O$ and $c_n>0$ tending to zero such that
		\begin{equation}\label{contradictionforcarathéodory}
			c_\O(z_n,w_n) \leq \log\left(1+{c_n A_\O(z_n,w_n)}\right).
		\end{equation}
		By passing to a subsequence, if necessary, we distinguish two cases.
		
		\textit{Case I.} The sequence $c_\O(z_n,w_n)$ is bounded from above.  Due to \eqref{comparasion} and the lower bound we obtained for the Kobayashi distance we have that $k_\O(z_n,w_n) \lesssim 1$, whence $A_\O(z_n,w_n) \lesssim 1$. Using \eqref{comparasion}, the lower bound we obtained for the Kobayashi distance and the fact that  $\log(1+r)\leq r \leq \log(1+C'r)$ for $0\leq r\leq r'$, we get  \begin{multline*}
			c_\O(z_n,w_n) \geq C^{-1}_1 k_\O(z_n,w_n) \geq  C^{-1}_1 \log  \left(1+{c A_\O(z_n,w_n)}\right) \geq  \\
			{c_2 A_\O(z_n,w_n))} \geq  \log\left(1+{c_2 A_\O(z_n,w_n)}\right) .
		\end{multline*}
		This contradicts \eqref{contradictionforcarathéodory}.

		\textit{Case II.} The sequence $c_\O(z_n,w_n)$ escapes to infinity.  Then, clearly, $k_\O(z_n,w_n)\to \infty,.$ Therefore, by the upper bound we obtained for the Kobayashi distance, $ A_\O(z_n,w_n)\to \infty$, as well. Using the lower bound we obtained for the Kobayashi distance, by  \eqref{comparasion} we get  \begin{multline*}
			c_\O(z_n,w_n) \geq k_\O(z_n,w_n) - C_1 \geq
			\log\left(1+{c A_\O(z_n,w_n)}\right) -C_1 \geq \\ \log\left({c A_\O(z_n,w_n)}\right) -C_1 \geq \log\left({c_3 A_\O(z_n,w_n)}\right) \geq \log\left(1+{\dfrac{c_3}{2} A_\O(z_n,w_n)}\right)
		\end{multline*} provided that $n$ is big enough.We again see that a sequence satisfying \eqref{contradictionforcarathéodory} cannot exist.
		
		\medskip

		\noindent\textit{Proof of the estimates of the Bergman distance.} It follows from the proof of \cite[Corollary 1.4]{NT} that if $\O$ has $\mathcal C^{3,1}$-smooth boundary, then there exists $C > 1$ such that
		$$ |k_\O(z,w)-b_\O(z,w)| \leq C g_\O(z,w) \:\:\:\:\:\text{and}\:\:\:\:\: C^{-1} \leq \frac{k_\O(z,w)}{b_\O(z,w)} \leq C .$$
		Subsequently, the estimates follow through the arguments we have presented for both the Lempert function and the Carath\'eodory distance.
		
		This finishes the proof of Theorem \ref{profnikolovsdream}.\end{proof}
	
	\begin{proof}[Proof of Corollary \ref{localizationforkobayashiandlempert}]
		Fix $\O$ and $p$ satisfying the assumptions. We may find two neighborhoods of $p$, $V\subset\subset U$ such that $\O\cap U$ is strongly pseudoconvex and for all $z,w \in \O\cap V$ the equality \begin{equation}A_\O(z,w)=A_{\O\cap U}(z,w)\label{equationforlocalization}\end{equation}
		holds. As $\O\cap U$ is strongly pseudoconvex, Theorem ~\ref{profnikolovsdream} shows that there exist $C>c>0$ such that we have
		\begin{equation}\label{estimatesforthesmalldomain} \log(1+c A_{\O\cap U}(z,w)) \leq k_{\O\cap U}(z,w) \leq l_{\O\cap U}(z,w) \leq \log(1+C A_{\O\cap U}(z,w)), \:\:\:\:\: z,w\in \O\cap V . \end{equation}
		Then \eqref{equationforlocalization} and the monotonicity of the Lempert function under the inclusion of sets give
		$$ k_\O(z,w) \leq l_\O(z,w) \leq l_{\O\cap U}(z,w) \leq \log(1+C A_\O(z,w)). \:\:\:\:\: z,w\in\O\cap V.$$
		
		It remains to show the lower bound. By shrinking $U$ and $V$ if necessary, \cite[Theorem 1.1]{NT2} shows that there exists a constant $C>1$ such that
		\begin{equation}\label{kobayashidistancecomparasion}
			k_{\O\cap U}(z,w) \leq k_\O(z,w)+C  \:\:\text{and}\:\: k_{\O\cap U}(z,w) \leq C k_\O(z,w), \:\:\:\:\: z,w \in \O\cap V. \end{equation}
		Having \eqref{equationforlocalization} and \eqref{estimatesforthesmalldomain} the lower bound becomes straightforward. It can be achieved by repeating the arguments provided in the proof of Theorem~\ref{profnikolovsdream} concerning the lower bound for the Carath{\'e}odory distance. The only difference is to use \eqref{kobayashidistancecomparasion} instead of \eqref{comparasion}.
	\end{proof}

	\section{Strictly linearly convex case}\label{strictlylinearlyconvexcase}
	
	Let $\O$ be a domain in $\mathbb{C}^N$. Set
	$$h_\O(z,w)=\|(z-w)_z\|+\|z-w\|\delta^{1/2}_\O(z).$$
	Corollary \ref{localizationforkobayashiandlempert} shows that if $p\in\partial\O$ is a $\mathcal C^{2,\alpha}$-smooth strongly pseudoconvex boundary point, then there exists $C>1$ such that
	\begin{equation}\label{exponentialform}C^{-1}<\frac{(e^{k_\O(z,w)}-1)\delta^{1/2}_\O(z)\delta^{1/2}_\O(w)}{h_\O(z,w)+\|z-w\|^2}<C,
		\quad z\neq w\in \O,\ z,w\mbox{ near } p.\end{equation} It is clear that none of the summands of $h_\O$ can be removed in such an estimate
	($C$ may vary). On the other hand, the proof of Lemma \ref{letshopethatthisisthefinaldifficulty} asserts that the terms $\|z-w\|^2$ in the upper estimates in Theorem \ref{profnikolovsdream} and Corollary \ref{localizationforkobayashiandlempert} are also essential. One can observe this fact directly. Indeed, consider a domain that near the origin is given by $D=\{\rho<0\}\subset \mathbb{C}^2$,  where $\rho(z)=\re(z_1-z_2^2)+|z_2|^2$. If we take $z_n,w_n\in D$ such that $z_n=(\delta_n,0), w_n=(0,\epsilon_n)$ with $\delta_n=o(\epsilon_n)$, we find that $\|(z_n-w_n)_{z_n}\|+\delta^{1/2}_D(z_n)\|z_n-w_n\|=o(\|z_n-w_n\|^2)$. Hence, the term $\|z-w\|^2$ in the lower bound cannot be estimated by the other two terms. In particular, it cannot be, in general, removed from the upper bound in Theorem \ref{profnikolovsdream} as well as in Corollary \ref{localizationforkobayashiandlempert}.

	Interestingly, the scenario changes when dealing with strictly linearly convex domains. This section is devoted to exploring this matter and uncovering some results that stand on their own significance.
	
	Let	$\tilde \delta_\O$ denote the signed boundary distance function of $\O$. Recall that $\O$ is called strictly linearly convex at a $\mathcal C^2$-smooth boundary point
	$p$ if the restriction of the Hessian of $\tilde \delta_\O$ to the complex tangent plane $T^{\C}_p(\partial \O)$ is strictly
	positive definite. We claim that the $\|z-w\|^2$ term in \eqref{exponentialform} is
	superfluous if and only if $\O$ is strictly linearly convex at $p\in\partial \O$. Precisely, our aim is to show the following:
		
	\begin{prop}\label{1} Let $p$ be a $\mathcal C^2$-smooth boundary point of a domain $\O$ in $\C^N.$
		The following conditions are equivalent:
		\begin{enumerate}
			\item $\O$ is strictly linearly convex at $p$;
			\item $\ds c_2=\liminf_{\O\ni z,w\to p}\frac{h_\O(z,w)}{\|z-w\|^2}>0;$
			\item $\ds c_3=\liminf_{\O\ni w\to p}\frac{\|(p-w)_p\|}{\|p-w\|^2}>0.$
		\end{enumerate}
	\end{prop}
	
	An immediate consequence of Proposition~\ref{1} is that $c_2=c_3=0$ if $\O$ is not strictly linearly convex at $p$. Also, if $\O$ is strictly linearly convex, then $\|(p-z)_p\|\neq 0$
	for any $p\in\partial \O,$ $z\in\overline \O,$ $z\neq p.$ A standard compactness argument yields:
	\begin{cor}\label{2} A bounded domain $\O$ in $\mathbb C^N$ with $\mathcal C^2$-smooth boundary
		is strictly linearly convex if and only if there exists $c>0$ such that
		$$h_\O(z,w)\ge c\|z-w\|^2,\quad z,w\in \O,\ z\mbox{ near }\partial \O.$$
	\end{cor}
	
	Set $$\ds c_4=\liminf_{\partial \O\ni w\to p}\frac{\|(p-w)_p\|}{\|p-w\|^2},$$ $G=\C^N\setminus\overline \O$ and let $\lambda=\lambda_{\O,p}$ denote the minimal eigenvalue of the restriction of the Hessian of $\tilde\delta_\O$ to $T_p^{\mathbb C}(\partial\O)$. Clearly, $\O$ is strictly linearly convex at $p$ if and only if $\lambda>0$.
	Proposition \ref{1} is a consequence of the following:
	\begin{prop}\label{3} $\:$
		\begin{enumerate}
			
			\item  If $\O$ is strictly linearly convex at $p$, then
			\begin{enumerate}
				\item $c_3=c_4=\lambda$,
				\item  $c_2=\min\{\lambda,\lambda^{1/2}\}.$
			\end{enumerate}
			\item If neither $\O$ nor $G$ is
			strictly linearly convex at $p,$ then $c_4=0.$
		\end{enumerate}
	\end{prop}
	
	Observe that Proposition \ref{3} shows that $c_4=\lambda_{G,p}$ if $G$ is strictly
	linearly convex at $p$.
	
	\begin{proof}[Proof of Proposition \ref{1}]
		
		$(1)\Rightarrow(2)$ follows from the first part of Proposition \ref{3}.
		
		$(2)\Rightarrow(3)$ follows from the fact that $c_2\le c_3.$
		
		$(3)\Rightarrow(1)$: We may assume that $p=0$ and
		\begin{equation}\label{u}
			\O\cap U=\{u=(x_1+iy_1,u')\in U:x_1>y_1g(y_1,u')+f(u')\}
		\end{equation}
		for some polydisc $U=U_1\times U'$ centered at $0,$ where $g(0)=0$ and
		$\mbox{ord}_0f\ge 2.$
		Let $u'\in U'.$ Then $(f(u'),u')\in\partial \O.$ Let $c_3'\in (0,c_3).$
		Shrink $U'$ if necessary to get to get $|f(u')|\ge c_3'\|u'\|^2$ on $U'$.
		Assume that $\O$ is not strictly linearly convex at $p$, then $f(u'_0)<0$ for some $u'_0\in U'_\ast$, where $U'_\ast=U'\setminus\{0'\}$. This implies that $f<0$ on ${U'_\ast}$ so $\{0\}\times U'_\ast\subset\O$. Hence $0=c_3=c_3'=0$,
		a contradiction. Thus $\O$ is strictly linearly convex
		at $p=0$ and $f(u')\ge c_3'|u'|^2$. \end{proof}
	\begin{proof}[Proof of Proposition \ref{3}]
		
		Let us start with (2). We may assume \eqref{u} holds.
		Then there exists a unit vector $v'\in\mathbb C^{N-1}$
		such that $$\lim_{\epsilon\to 0}\frac{f(\epsilon v')}{\epsilon^2}=0.$$ Set
		$w_\epsilon=(f(\epsilon v'),\epsilon v').$ Then $w_\epsilon\in\partial \O$ for $\epsilon\in(0,\epsilon_0)$ and
		$$\lim_{\epsilon\to 0}\frac{\|{(w_\epsilon)_0}\|}{\|w_\epsilon\|^2}=0.$$
		
		Now we shall prove (1)(a). It follows from the proof of Proposition \ref{1} that $c_4\le\lambda.$
		Since $c_3\le c_4,$ it remains to show that $\lambda\le c_3.$ Fix $t\in(0,1).$ Shrinking $U$, if necessary, we may assume that \eqref{u} holds,
		$f(u')\ge t\lambda|u'|^2$ and $|g|<1-t.$ Setting $w=(x_1+iy_1,w')$ we see that
		$$|w_1|>t|x_1|+(1-t)|y_1|\ge t^2\lambda|w'|^2.$$
		Then $w\in\O$ and $c_3\ge t^2\lambda$. Letting $t\to 1$ gives $c_3\ge\lambda.$
		\smallskip
		
		We are left with (1)(b). We may assume that $\eqref{u}$ holds. There exists a unit vector $v'\in\C^{N-1}$
		such that $$\lim_{\epsilon\to 0}\frac{f(\epsilon v')}{\epsilon ^2}=\lambda.$$ Set $z_\epsilon=((1+\epsilon)f(\epsilon v'),0')$
		and $w_\epsilon=((1+\epsilon)f(\epsilon v'),\epsilon v').$ Then $z_\epsilon, w_\epsilon\in \O$ for $\epsilon\in(0,\epsilon_0)$ and
		$$\lim_{\epsilon\to 0}\frac{b_\O(z_\epsilon,w_\epsilon)}{\|z_\epsilon-w_\epsilon\|^2}=\lambda^{1/2}. $$
		Consequently, $c_2\le\lambda^{1/2}.$ Since $c_2\le c_3=\lambda,$ it remains to show that
		$c_2\ge\min\{\lambda,\lambda^{1/2}\}.$
		
		We shall proceed similarly as before.
		Fix $t\in(0,1).$ For any $z\in \O$ close to $p,$ we may assume
		that $\pi_\O(z)=0,$ where $\pi_\O$ is the projection onto the boundary, and that \eqref{u} holds, $f(u')=f_z(u')\ge t\lambda\|u'\|^2,$
		$g=g_z,$ $\|g\|<1-t,$ and $U=U_t$ does not depend on $z.$ Then $z=(x,0'),$
		$x>0.$ Setting $w=(x_1+iy_1,w'),$ we have that
		$$\|x_1\|+(1-t)\|y_1\|\ge t\lambda\|w'\|^2,$$
		$$h_\O(z,w)\ge t|x-x_1|+(1-t)|y_1|+\sqrt{x}|w'|.$$
		
		\noindent{\it Case I.} $t^2x_1<|w'|^2.$ Then
		$$h_\O(z,w)\ge t|x_1|+(1-t)|y_1|>t^2\lambda|w'|^2.$$
		
		\noindent{\it Case II.} $t^2x_1\ge |w'|^2.$ Then
		$$h_\O(z,w)\ge (1-t)|y_1|+\sqrt{x_1}|w'|.$$
		
		{\it II.I.} If $x_1\ge\lambda|w'|^2,$ then $h_\O(z,w)\ge\lambda^{1/2}|w'|.$
		
		{\it II.II.} If $x_1<\lambda|w'|^2,$ then $\lambda>1$ and hence
		$$h_\O(z,w)>(1-t)|y_1|+x_1/\lambda^{1/2}\ge t\lambda^{1/2}|w'|.$$
		Now we easily conclude that $c_2\ge\min\{\lambda,\lambda^{1/2}\}.$
	\end{proof}
	
	Let $v_z^{\mathbb{R}}$ denote the real normal component of the vector $v$ taken with respect to $\pi_\O(z)\in\partial\O$ and set $$h^\mathbb{R}_\O(z,w)=\|(z-w)^\mathbb{R}_z\|+\|z-w\|\delta^{1/2}_\O(z).$$ Proposition \ref{1} has the following obvious counterpart in $\mathbb{R}^N$ (with simpler proofs):
	
	\begin{prop} Let $p$ be a $\mathcal C^2$-smooth boundary point of a domain $\O$ in $\C^N.$
		The following conditions are equivalent:
		\begin{enumerate}
			\item $\O$ is strongly convex at $p$;
			\item $\displaystyle{\liminf_{\O\ni z,w\to p}
				\frac{h^\mathbb{R}_\O(z,w)}{\|z-w\|^2}>0};$
			\item $\displaystyle{\liminf_{\O\ni w\to p}\frac{\|(p-w)_p^{\mathbb{R}}\|}{\|p-w\|^2}>0}.$
		\end{enumerate}
	\end{prop}

	\medskip
	
	\noindent
	
	\textbf{Acknowledgements.} The authors express their gratitude to Pascal J. Thomas for valuable conversations and comments which led to the improvement of this note.
	
	The first named author was supported by the grants SONATA BIS no.
	2017/26/E/ST1/ 00723 and Sheng no. 2023/48/Q/ST1/00048 of the National Science Center, Poland. The second named author was partially supported by the
	Bulgarian National Science Fund, Ministry of Education and Science of Bulgaria
	under contract KP-06-N52/3. The third author received support from the University Research School EUR-MINT (State support managed
	by the National Research Agency for Future Investments program bearing the reference ANR-18-EURE-0023). The second and third-named authors enjoy the support of
	the PHC Rila program 48135TJ, which made possible the stay of the third author at the Institute for
	Mathematics and Informatics of the Bulgarian Academy of Sciences, Sofia, during
	which this work was started.
	Part of this work was done during the visit of the third named author in Jagiellonian University, Krakow which was possible due to the aforementioned grants of the first and the third named authors. This paper forms a part of the Ph.D. work of the third named author, under the
	direction of P. J. Thomas and N. Nikolov, in Toulouse I. of M. The third
	named author is currently partially supported by the MIUR Excellence Department
	Project 2023-2027 MatMod@Tov awarded to the Department of Mathematics, University of Rome Tor Vergata, by PRIN Real and Complex Manifolds: Topology, Geometry
	and holomorphic dynamics n.2017JZ2SW5 and by GNSAGA of INdAM.
	
	{}
	

\begin{thebibliography}{999999}
		
		\bibitem[BB]{BB} Z. M. Balogh, M. Bonk, {\it Gromov hyperbolicity and the Kobayashi metric on
			strictly pseudoconvex domains}, Comment. Math. Helv. 75 (2000), 504-533.
		
		\bibitem[BFW]{BFW} F. Bracci, J.E. Fornaess, E.F. Wold, {\it Comparison of invariant metrics and distances on strongly pseudoconvex domains and worm domains,} Math. Z. 292 (2019),
		879–893.
		
		\bibitem[BGNT]{BGNT} F. Bracci, H. Gaussier, N. Nikolov, P. J. Thomas, {\it Local and global visibility and Gromov hyperbolicity of domains with respect to the Kobayashi distance}, Trans. Am. Math. Soc. 377 (2024) 471-493.
		
		\bibitem[BST]{BST} F. Bracci, A. Saracco, S. Trapani, {\it The pluricomplex Poisson kernel for strongly pseudoconvex domains}, Adv. Math. 380 (2021), article 107577.
		
		
		\bibitem[BH]{BH} M. R. Bridson, A. Haefliger, {\it Metric spaces of non-positive curvature}, Vol. 319. Springer Science $\&$ Business Media, 2013.
		
		\bibitem[BZ]{BZ} G. Bharali, A. Zimmer, {\it Goldilocks domains, a weak notion of visibility, and	applications}, Adv. Math. 310 (2017), 377–425.
			
		\bibitem[CHL]{CHL} C. H. Chang, M. C. Hu, H.P. Lee, {\it Extremal analytic discs with prescribed boundary data}, Trans. Amer. Math. Soc. 310 (1988), 355–369.
		
		
		\bibitem[DGZ]{DGZ} F. Deng, Q. Guan, L. Zhang, {\it Properties of squeezing functions and global transformations of bounded domains}, Trans. Amer. Math. Soc. 368 (2016), 2679–2696.
				
		\bibitem[FM]{FM} B. L. Fridman, D. Ma, {\it On exhaustion of domains}, Indiana University Mathematics Journal, (1995) 385-395.
		
		\bibitem[Hua1]{H} X. Huang, {\it A preservation principle of extremal mappings near a strongly pseudoconvex point and its applications}, Illinois J. Math. 38 (1994), 283–302.
		
		\bibitem[Hua2]{H2}  X. Huang, {\it A non-degeneracy property of extremal mappings and iterates of holomorphic self-mappings}, Ann. Sc. Norm. Super. Pisa Cl. Sci. (4) Vol. XXI (1994),
		399–419.
		
		\bibitem[Hua3]{H3} X. Huang, {\it Revisitting a non-degeneracy property for extremal mappings}, Acta
		Math. Sci. 41B (2021), 1829–1838.
		
		\bibitem[JP]{JP} M. Jarnicki, P. Pflug, {\it Invariant distances and metrics in complex analysis,}
		2nd extended ed., de Gruyter Expo. Math., vol. 9, Walter de Gruyter, Berlin, 2013.
			
		\bibitem[KN]{KN}  {\L}. Kosi{\'n}ski, N. Nikolov, {\it Lower estimates of the Kobayashi distance and limits
			of complex geodesics}, Math. Ann. 389 (2024) 1925-1937. 
		
		\bibitem[KNT]{KNT} {\L}. Kosi{\'n}ski, N. Nikolov, P. J. Thomas, {\it A Gehring-Hayman inequality for strongly pseudoconvex domains}, Int. Math. Res. Not. 2024(11) (2024) 9165-9177.
		
		\bibitem[Lem1]{L1} L. Lempert, {\it La m\'etrique de Kobayashi et la repr\'esentation des domains
			sur la boule}, Bull. Soc. Math. France 109 (1981), 427--474.
		
		\bibitem[Lem2]{L2} L. Lempert, {\it Intrinsic distances and holomorphic retracts}, Complex Analysis
		and Applications '81, Sofia (1984), 341--364.
		
		\bibitem[Lem3]{L} L. Lempert, { Erratum: ”A precise result on the boundary regularity of biholomorphic mappings [Math. Z. 193 (1986), 559–579;]\it}, Math. Z.	206 (1991), 501–504.
		
		
		\bibitem[Ma]{M} D. Ma, {\it Sharp estimates of the Kobayashi metric near strongly pseudoconvex points}, Contemp. Math, 137 (1992), 329-338.
		
		\bibitem[N{\"O}]{NO} N. Nikolov, A. Y. {\"O}kten, {\it Strongly Goldilocks domains, quantitative visibility, and applications}, J. Math. Anal. Appl. 534 (2024) 128130.
		
		\bibitem[Nik]{N} N. Nikolov, {\it Comparison and localization of invariant functions on strongly pseudoconvex domains}, Bull. London Math. Soc. 55 (2023), 2052-2061.
		
		
		\bibitem[NA]{NA} N. Nikolov, L. Andreev, { \it Estimates of the Kobayashi and quasi-hyperbolic distances}, Ann. Mat. Pura App. 196 (2017), 43-50.
			
		\bibitem[NT1]{NT} N. Nikolov, P.J. Thomas, { \it Comparison of the real and the complex Green functions, and sharp estimates of the Kobayashi distance}, Ann. Sc. Norm. Super. Pisa Cl. Sci. (5) Vol. XVIII (2018), 1125–1143.
		
		\bibitem[NT2]{NT2} N. Nikolov, P.J. Thomas, {\it Quantitative localization and comparison of invariant distances of domains in $\mathbb{C}^n$},
		J. Geom. Anal. 33 (2023), article 35.
		
		\bibitem[Pan]{P} M. Y. Pang, {\it On infinitesimal behavior of the Kobayashi distance}, Pacific J.	Math. 162 (1994), 121–141.
		
		\bibitem[Ven]{V} S. Venturini, {\it Pseudodistances and pseudometrics on real and complex manifolds}, Ann. Mat. Pura App. 154 (1989), 385--402.
		
		\bibitem[Zim]{Z} A. Zimmer, {\it Gromov hyperbolicity and the Kobayashi metric on convex domains of finite type}, Math. Ann. 365 (2016), 1425–1498.
		
		
		
		
		
		
	\end{thebibliography}
\end{document}